\documentclass[12pt]{amsart}
\usepackage{graphicx,amssymb,amscd}
\usepackage{longtable}
\usepackage{bm}
\usepackage{pict2e}
\usepackage[dvips]{color}

\usepackage{times}

\newtheorem{theorem}{Theorem}[section]
\newtheorem{proposition}[theorem]{Proposition}
\newtheorem{lemma}[theorem]{Lemma}

\newtheorem{corollary}[theorem]{Corollary}
\theoremstyle{definition}
\newtheorem{definition}[theorem]{Definition}
\newtheorem{example}[theorem]{Example}

\newcommand{\bdd}{\mbox{$\partial$}}

\theoremstyle{remark}
\newtheorem{remark}[theorem]{Remark}

\theoremstyle{problem}
\newtheorem{problem}[theorem]{Problem}

\numberwithin{equation}{section}

\begin{document}

\title[Hyperbolic volume and twisted Alexander invariants of knots and links]
{Hyperbolic volume and twisted Alexander invariants of knots and links} 

%\date{\today}


\author{Hiroshi Goda}
\address{Department of Mathematics,
Tokyo University of Agriculture and Technology,
2-24-16 Naka-cho, Koganei,
Tokyo 184-8588, Japan}
\email{goda@cc.tuat.ac.jp}

\begin{abstract}
Let $\Delta_{L,\rho_n}(t)$ be the twisted Alexander polynomial with respect to 
the representation given by the composition of the lift of the holonomy 
representation of a certain hyperbolic link $L$ and the $n$-dimensional 
irreducible complex representation of $\text{SL}(2,\mathbb C)$. 
We consider a sequence of $\Delta_{L,\rho_n}(t)$ and extract 
the volume of the complement of $L$ from the asymptotic behaviour 
of the sequence obtained by evaluating $t=1$ or $t=-1$. 
\end{abstract}

\thanks{
This work was supported by JSPS KAKENHI 
Grant Number JP15K04868.
}


\maketitle

\section{Introduction}
One of the known classical knot and link invariants 
is the Alexander polynomial \cite{alexander}. 
This polynomial is derived from the fundamental group 
of the complement of a knot or link and is determined 
by its maximal metabelian quotient. 
This computable invariant plays an important role in the knot theory. 
In 1990, Lin \cite{lin} introduced the twisted Alexander polynomial 
associated with a knot $K$ and a representation of 
the knot group $\pi_1(S^3\setminus K)\to \text{GL}(n,\mathbb F)$, 
where $\mathbb F$ is a field.
Wada \cite{wada} generalized this work and showed how to define 
a twisted polynomial given only a presentation of a group and representations.

Kirk and Livingston \cite{kirk-livingston} gave a construction 
of the twisted Alexander invariants, 
which they interpreted in the general framework of Reidemeister torsion, 
and showed how the work of Lin and Wada could also be interpreted in that framework. 
We note that these invariants can be calculated using the classical Fox calculus. 
Since then, various applications of the twisted Alexander invariants 
have been given.
For example, see the survey papers of \cite{friedl-vidussi, morifuji}.

A 3-manifold $M$ is called {\it hyperbolic} if the interior admits a complete 
metric of constant curvature $-1$ of finite volume. 
If $M$ is hyperbolic, it follows from Mostow-Prasad rigidity 
that the hyperbolic metric is unique up to isometry. 
Furthermore, its fundamental group $\pi_1(M)$ dominates all geodesic 
information of $M$. 
Therefore, the volume of $M$ with respect to the hyperbolic metric 
is an invariant of $M$ and may be computed theoretically 
from a presentation of $\pi_1(M)$.

We say that a knot $K$ is {\it hyperbolic} if $S^3\setminus K$ 
is hyperbolic. In that case, we will denote the volume by
$\text{Vol}(K)$. 
Thurston's geometrization theorem states that any knot that 
is neither a torus knot nor a satellite knot is a hyperbolic knot. 
The link case is the same.

In this paper, we give a volume formula of hyperbolic knots or links 
using twisted Alexander invariants. 
The volume formula in the knot case has already been proven 
\cite{goda}.

M\"{u}ller \cite{muller2} 
established a volume formula for a closed hyperbolic 
3-manifold using Ray-Singer analytic torsion.  
The Selberg trace formula was applied to the upper half 3-space, 
so that he obtained an expression of Ray-Singer analytic torsion 
in terms of the twisted Ruelle zeta function. 
The volume formula is based on the expression and functional equations 
for the twisted Selberg and Ruelle zeta functions associated with 
the $n$-th symmetric power of the standard representation of SL$(2,\mathbb C)$, 
which include the volume information of a closed hyperbolic 3-manifold. 
By \cite{muller1} of M\"{u}ller on the equivalence 
between Ray-Singer torsion and  Reidemeister torsion 
for unimodular representations, 
we know that the hyperbolic volume of a closed 3-manifold can be 
expressed using Reidemeister torsion. 

Applying Thurston's hyperbolic Dehn surgery theorem, 
Menal-Ferrer and Porti \cite{menal-porti2} obtained 
a formula for the volume of a cusped hyperbolic 3-manifold $M$ 
using the so-called higher-dimensional Reidemeister torsion invariants, 
which are associated with representations 
$\rho_n:\pi_1(M)\to\text{SL}(n,\mathbb C)$
corresponding to the holonomy representation 
$\text{Hol}_M:\pi_1(M)\to \text{PSL}(2,\mathbb C)$
(see Section \ref{sec:holonomy} for details).
In the present paper, we consider the relationship
between the higher-dimensional Reidemeister torsion invariants 
and the twisted Alexander invariants, 
thereby obtaining a formula for the volume for a hyperbolic link
using twisted Alexander invariants.

Let $L=K_1\cup \cdots \cup K_b$ be a $b$-component link in $S^3$. 
We call $L$ an \textit{algebraically split link} if 
the linking numbers $\ell k(K_i,K_j)$ are zero for all $i,\,j$ 
(Definition \ref{def:link}).
Any knot may be regarded as an algebraically split link. 
Let $\Delta_{L,\rho_{n}}(t)$ be the twisted Alexander invariant 
of Wada's notation \cite{wada}.
For an integer $k (>1)$, set
$\displaystyle{
\mathcal A_{L,2k}(t)
=
\frac{\Delta_{L,\rho_{2k}}(t)}{\Delta_{L,\rho_2}(t)}}
$
and 
$\displaystyle{
\mathcal A_{L,2k+1}(t)
=
\frac{\Delta_{L,\rho_{2k+1}}(t)}{\Delta_{L,\rho_3}(t)}}.
$
A result of this paper is the following asymptotic formula.

\begin{theorem}\label{thm:link}
Let $L$ be an algebraically split hyperbolic link in the 3-sphere. 
Then, 
$$
\lim_{k\to\infty}\frac{\log|\mathcal A_{L,2k+1}(1)|}{(2k+1)^2}
=
\lim_{k\to\infty}\frac{\log|\mathcal A_{L,2k}(1)|}{(2k)^2}
=
\frac{{\rm Vol}(L)}{4\pi}.
$$
\end{theorem}

It might be interesting to compare the volume conjecture 
on the colored Jones polynomial for knots in \cite{kashaev, murakami}
(see Problem \ref{pro:Jones}).
The main theorem (Theorem \ref{thm:main}) provides 
a generalization of this formula to a complete, oriented, 
hyperbolic 3-manifold whose boundary consists of 
torus cusps. 
The equation obtained by evaluating $-1$ instead of $1$ 
also holds (Corollary \ref{cor:-1}).

This paper is organized as follows. 
In Section \ref{sec:pre}, we recall the definitions of Reidemeister 
torsions and twisted Alexander invariants and then fix the notation.
Section \ref{sec:holonomy} is devoted to explaining how to 
obtain the $n$-dimensional complex irreducible representation 
$\pi_1(M)\to\text{SL}(n,\mathbb C)$ 
via the holonomy representation of $\pi_1(M)$.
The important results used to prove the main theorem are those of 
Menal-Ferrer and Porti \cite{menal-porti1,menal-porti2} 
and a slight generalization of a result of Dubois and Yamaguchi 
\cite{dubois-yamaguchi1}.
We review the results of Menal-Ferrer and Porti in Section \ref{sec:menal-porti}
and discuss the generalization of the results of Dubois and Yamaguchi 
in Sections \ref{sec:dubois} and \ref{sec:yamaguchi}.
Section \ref{sec:results} is devoted to the proof of the main results. 
Theorem \ref{thm:link} means that we are able to obtain an approximate value of 
the volume of an algebraically split link via Fox calculus. 
We give some sample calculations in Section \ref{sec:calculation}.

The author would like to thank Professor Yoshikazu Yamaguchi 
for many helpful conversations
and Professor Joan Porti for his advice.

\section{Preliminaries}\label{sec:pre}
In this section, we review the basic notion and results regarding Reidemeister torsion 
\cite{dubois-yamaguchi1,porti2,turaev2}.

Let $C_*=(0\to C_n\xrightarrow{d_n}C_{n-1}\xrightarrow{d_{n-1}}\cdots
\xrightarrow{d_1}C_0\to 0)$ be a chain complex of finite-dimensional vector spaces 
over a field $\mathbb F$.
Choose a basis $\bold c^{i}$ of $C_i$ and a basis $\bold h^i$ of the $i$th homology 
group $H_i(C_*)$ if nonzero.
We denote by $\bold c$ ($\bold h$, resp.) the collection $\{\bold c^i\}$ 
($\{\bold h^i\}$, resp.). 
The Reidemeister torsion of $C_*$ with this choice of bases is defined as follows. 

For each $i$, let $\bold b^i$ be a set of vectors in $C_i$ 
such that $d_i(\bold b^i)$ is a basis of $B_{i-1}={\rm im}(d_i:C_i\to C_{i-1})$, 
and $\bold{\tilde{h}}^i$ a lift of $\bold h^i$ in $Z_i=\ker(d_i:C_i\to C_{i-1})$.
The set of vectors $d_{i+1}(\bold b^{i+1})\bold{\tilde{h}}^i\bold b^i$ 
becomes a basis of $C_i$. 
Let $[d_{i+1}(\bold b^{i+1})\bold{\tilde{h}}^i\bold b^i/\bold c^i]\,(\in\mathbb F^*)$ 
be the determinant of the transformation matrix between those bases. 
\begin{definition}\label{def:torsion}
The {\it Reidemeister torsion} of the chain complex $C_*$ 
with bases $\bold c$ and $\bold h$ is
$$
\text{Tor}(C_*,\bold c,\bold h)
=
\prod_{i=0}^n[d_{i+1}(\bold b^{i+1})\bold{\tilde{h}}^i\bold b^i/\bold c^i]^{(-1)^{i+1}}
\hspace{0.3cm}
\in\mathbb F^*/\{\pm 1\}.$$
\end{definition}
It is known that $\text{Tor}(C_*,\bold c,\bold h)$ is independent of the choice of 
$\bold b^{i}$ and the lifts $\bold{\tilde{h}}^i$.
We are interested in the volume, namely, 
the absolute values of the Reidemeister torsion, 
so we do not consider its sign in this paper.   
Thus, the torsion lies in $\mathbb F^*/\{\pm 1\}$, 
but there are ways to avoid the sign indeterminacy.
\begin{remark}\label{rmk:sign}
In \cite{menal-porti2}, the authors use the power $(-1)^i$ instead of $(-1)^{i+1}$ 
in Definition \ref{def:torsion}.
In that case, the sign of the right-hand side of the equation in Theorem 7.1 in \cite{menal-porti2}
is opposite. See Remark 2.2 and Theorem 4.5 in \cite{porti2}.
\end{remark}

Let $W$ be a finite CW-complex and 
$(V,\rho)$ be a pair of a vector space over $\mathbb F$ 
and a homomorphism of $\pi_1(W)$ into $\text{Aut}(V)$.
The vector space $V$ turns into a right $\mathbb Z[\pi_1(W)]$-module 
by using the right action of $\pi_1(W)$ on $V$ given by 
\begin{equation}\label{eq:module}
v\cdot \gamma=\rho^{-1}(\gamma)(v)
\end{equation}
for $v\in V$ and $\gamma\in\pi_1(W)$.
We denote this module by $V_{\rho}$. 
The complex of the universal cover of $W$ with integer coefficients 
$C_*(\widetilde{W};\mathbb Z)$ becomes a left $\mathbb Z[\pi_1(W)]$-module 
by the action of $\pi_1(W)$ on $\widetilde{W}$ as the covering transformation group.
We define the $V_{\rho}$-twisted chain complex of $W$ as follows: 
\begin{equation}\label{eq:chain1}
C_*(W;V_{\rho})=V_{\rho}\otimes_{\mathbb Z[\pi_1(W)]}C_*(\widetilde{W};\mathbb Z).
\end{equation}
The boundary operator of the chain complex is defined by linearly 
and
\begin{equation}\label{eq:bdd1}
d_i(v\otimes c_i)
=
(\text{Id}\otimes d_i)(v\otimes c_i)=v\otimes d_i(c_i)
\end{equation}
for $c_i\in C_i(\widetilde{W};\mathbb Z)$.
Then, we can define the {\it $V_{\rho}$-twisted homology of $W$},  
which we denote by $H_*(W;V_{\rho})$.
 
Let $\{v_1,\ldots , v_n\}$ be a basis of $V$ 
and 
let $c_1^i,\ldots , c^i_{k_i}$ denote the set of $i$-dimensional cells of $W$. 
We take a lift $\tilde{c}^i_{j}$ of the cell $c^i_j$ in $\widetilde{W}$.
Then, for each $i$, 
$\bold{\tilde{c}}^i$ is a basis of the $\mathbb Z[\pi_1(W)]$-module $C_i(\widetilde{W};\mathbb Z)$. 
Thus, we have the following basis of $C_i(W;V_{\rho}):$
$$
\bold{\tilde{c}}^i
=
\{v_1\otimes \tilde{c}_1^i,\ldots,v_n\otimes\tilde{c}_1^i,\ldots ,
v_1\otimes\tilde{c}_{k_i}^i,\ldots,v_n\otimes\tilde{c}_{k_i}^i\}.
$$
If $H_i(W;V_{\rho})\neq 0$, choosing for each $i$ a basis $\bold h^i$ 
of $H_i(W;V_{\rho})$
and set $\bold h=\{\bold h^i\}$, 
we may define the Reidemeister torsion as 
$$
\text{Tor}(C_*(W;V_{\rho}),\bold c,\bold h)
\hspace{0.3cm}
\in\mathbb F^*/\{\pm 1\}.$$

\begin{remark}
The torsion does not depend on the lifts of the cells $\tilde{c}^i_j$ 
if $\det\rho=1$. Furthermore, if the Euler characteristic of $W$, say $\chi(W)$, 
is equal to $0$, we can use any basis of $V$ since the torsion is multiplied 
by the determinant of the base change matrix to the power $\chi(W)$.
\end{remark}

Here, we introduce a twisted chain complex with one variable.
This will be done using a $\mathbb Z[\pi_1(W)]$-module with the variable.
We regard $\mathbb Z$ as the multiplicative group generated by $t$, namely, 
$\mathbb Z=\langle t\rangle$
and consider a surjective homomorphism $\alpha: \pi_1(W)\to \mathbb Z$.
Every homomorphism from $\pi_1(W)$ to an Abelian group 
factors through the abelianization $H_1(W;\mathbb Z)$ of $\pi_1(W)$; 
that is, we have the following commutative diagram: 
$$\begin{picture}(200,70)
\linethickness{0.5pt}
\put(20,60){$\pi_1(W)$}
\put(55,63){\vector(1,0){70}}
\put(130,60){$H_1(W;\mathbb Z)$}
\put(150,55){\vector(0,-1){35}}
\put(145,8){$\mathbb Z$}
\put(55,55){\vector(2,-1){80}}
\put(85,27){$\alpha$}
\put(155,37){$\alpha_h$}
\end{picture}$$
Since we suppose $\alpha$ is a surjective, 
the induced homomorphism $\alpha_{h}$ is also surjective. 
In this paper, we consider one variable
rational functions $\mathbb F(t)$ 
associated with $\alpha_h$.
We write $V(t)$ 
for $\mathbb F(t)\otimes_{\mathbb F}V$ for brevity.
The fundamental group $\pi_1(W)$ acts on $V(t)$:
\begin{equation}\label{eq:rep}
\alpha\otimes\rho:
\pi_1(W)\to \text{Aut}(V(t)).
\end{equation}
Thus, $V(t)$ inherits the structure of a right $\mathbb Z[\pi_1(W)]$-module
as (\ref{eq:module}), 
so we denote it by $V_{\rho}(t)$
and consider the associated twisted chain complex 
$C_*(W;V_{\rho}(t))$ given by
\begin{equation}\label{eq:chain2}
C_*(W;V_{\rho}(t))=V_{\rho}(t)\otimes_{\mathbb Z[\pi_1(W)]}C_*(\widetilde{W};\mathbb Z)
\end{equation}
where
\begin{equation}\label{eq:module2}
f\otimes v\otimes \gamma \cdot c_*
=
f\alpha(\gamma)\otimes\rho^{-1}(\gamma)(v)\otimes c_*
\end{equation}
for any $f\in\mathbb F(t), v\in V, \gamma\in\pi_1(W)$ and 
$c_* \in C_*(\widetilde{W};\mathbb Z)$. 
We call this chain complex the {\it $V_{\rho}(t)$-twisted chain complex}
and its homology is denoted by $H_*(W;V_{\rho}(t))$.
Note that the boundary operator is defined by linearity 
and 
\begin{equation}\label{eq:bdd2}
d_i(f\otimes v\otimes c_i)
=
f\otimes v \otimes d_i(c_i)
\end{equation}
as in (\ref{eq:bdd1}).
If the twisted homology 
$H_*(W;V_{\rho}(t))=0$, 
then $C_*(W;V_{\rho}(t))$ is said to be {\it acyclic.} 

\begin{definition}\label{def:TAI}
If $C_*(W;V_{\rho}(t))$ is acyclic, then 
the {\it twisted Alexander invariant} of $W$ is
$$
\Delta_{W,\rho}(t)
=
\text{Tor}(C_*(W;V_{\rho}(t)), 1\otimes\bold c,\emptyset)
\hspace{0.3cm}
\in\mathbb F^*(t)/\{\pm 1\}.$$
\end{definition}

Proposition \ref{prop:homology1} asserts that 
the assumption holds in our setting.
Note that the twisted Alexander invariant is determined up to 
a factor $\pm t^{p}\,\,(p\in\mathbb Z)$ 
like the classical Alexander polynomial
(cf. Theorem 2 in \cite{wada}).

\begin{example}
Let $K$ be a knot in the 3-sphere $S^3$ 
and $W_K$ be a CW-complex corresponding to 
$S^3-\text{Int}N(K)$. 
If the representation $\rho \in \text{Hom}(\pi_1(W_{K});\mathbb C)$ 
is the trivial homomorphism and $\alpha$ is the abelianization of $\pi_1(W_K)$, 
that is, 
$\alpha:\pi_1(W_K)\to H_1(W_K;\mathbb Z)\cong \langle t\rangle$, 
then 
the twisted chain complex $C_*(W_K;V_{\rho}(t))$ is acyclic 
and the twisted Alexander invariant 
$\Delta_{W_K,\rho}(t)$ is the classical Alexander invariant 
divided by $(t-1)$.
See Theorem 4 in \cite{milnor1} or Capter II in \cite{turaev2}, for example.
\end{example}

In the present paper, we study a compact 3-manifold $M$ whose boundary 
consists of tori. 
In particular, we discuss a link complement in the 3-sphere. 
We always suppose a CW-structure $W$ of $M$ 
and use the notation 
$C_*(M;V_{\rho})=C_*(W;V_{\rho})$.
Furthermore, 
we omit the collection $\bold c$ of the basis and 
denote by $\text{Tor}(C_*(M;V_{\rho}),\bold h)$
under the simple homotopy invariance 
(see \cite{chapman} or Theorem 11.32 in \cite{davis-kirk}).
Following these, 
we denote by $\Delta_{M,\rho}(t)$ 
the twisted Alexander invariant for $M$ 
and 
by $\Delta_{L,\rho}(t)$ for a link $L$ 
in $S^3$.

\section{Holonomy representations of the fundamental groups of hyperbolic 
3-manifolds and their lifts}\label{sec:holonomy}

Let $M$ be a connected, oriented, hyperbolic 3-manifold. 
The hyperbolic structure of $M$ yields the holonomy representation: 
$\text{Hol}_M : \pi_1(M)\to \text{Isom}^+\mathbb H^3$, 
where $\text{Isom}^+\mathbb H^3$ is the orientation-preserving isometry 
group of the hyperbolic 3-space $\mathbb H^3.$
Using the upper half-space model, $\text{Isom}^+\mathbb H^3$ 
is identified with 
$\text{PSL}(2,\mathbb C)=\text{SL}(2,\mathbb C)/\{\pm 1\}$. 
Thurston \cite{thurston} proved that 
$\text{Hol}_M$ can be lifted to $\text{SL}(2,\mathbb C)$.

There were two ways to obtain a linear representation 
so that we can consider the twisted Alexander invariants: 
either
(i) compose the holonomy representation $\text{Hol}_M$ with the adjoint representation 
    to obtain $\pi_1(M)\to \text{Aut }(\frak s\frak l_2(\mathbb C))\le \text{SL}(3,\mathbb C)$
    or 
(ii) lift $\text{Hol}_M$ to a representation $\pi_1(M)\to \text{SL}(2,\mathbb C)$. 
The former approach is the focus of \cite{dubois-yamaguchi1} and 
the latter is that of \cite{dunfield-friedl-jackson}.
In particular, the open problem 1 in Section 1.7 of \cite{dunfield-friedl-jackson} 
asks whether the twisted Alexander polynomial
associated with (ii) determines the volume of the complement 
of a knot.  
The results in the present paper might be regarded as an answer to that question.

For all $n>0$, 
there exists a unique (up to isomorphism) $n$-dimensional complex, 
irreducible representation of $\text{SL}(2,\mathbb C)$, say 
$$
\sigma_n:\text{SL}(2,\mathbb C)\to\text{SL}(n,\mathbb C).
$$
Our approach in this paper is to use this $n$-dimensional representation 
of $M$ as the composition of a lift of ${\rm Hol}_M$ with $\sigma_n$ according to 
\cite{menal-porti2}. 
Thus, some relevant results in the case of $n=2$ or $3$ 
can be found in \cite{dubois-yamaguchi1} and \cite{dunfield-friedl-jackson}.
In Section 1.5 of \cite{dunfield-friedl-jackson}, 
the fact that twisted Alexander invariants 
associated with representations (i) and (ii) 
behave differently is deemed very mysterious.
However, 
it might be due to nothing more than  
the difference between the representations $\sigma_n$ ($n$: even)
and $\sigma_n$ ($n$: odd).

Let us give a more precise description of the representations.
By Propositions 2.1 and 2.2 in \cite{menal-porti2}, 
there is a one-to-one correspondence between the lifts of $\text{Hol}_M$ 
and the spin structures on $M$. 
Thus, attached to a fixed spin structure $\eta$ on $M$, 
we have a representation :
$$
\text{Hol}_{(M,\eta)}: \pi_1(M,\eta)\to\text{SL}(2,\mathbb C).
$$
The vector space $\mathbb C^2$ has the standard action of $\text{SL}(2,\mathbb C)$.
It is known that the pair of the symmetric product $\text{Sym}^{n-1}(\mathbb C^2)$ 
and the induced action by $\text{SL}(2,\mathbb C)$ gives an $n$-dimensional 
irreducible representation of $\text{SL}(2,\mathbb C)$.
Concretely, 
let $V_n$ be the vector space of homogeneous polynomials on $\mathbb C^2$ 
with degree $n-1$; namely, 
$$
V_n
=
\text{span}_{\mathbb C}
\langle
x^{n-1},x^{n-2}y, \ldots , xy^{n-2},y^{n-1}
\rangle.
$$
Then, the symmetric product $\text{Sym}^{n-1}(\mathbb C^2)$ 
can be identified with $V_n$ and the action of $A\in\text{SL}(2,\mathbb C)$ 
is expressed as
$$
A\cdot p(x,y)
=
p(x',y')
\hspace{0.5cm}
\text{where}
\begin{pmatrix}
x' \\
y'
\end{pmatrix}
=
A^{-1}
\begin{pmatrix}
x \\
y
\end{pmatrix}
$$
where $p(x,y)$ is a homogeneous polynomial.
We denote by $(V_n,\sigma_n)$ the representation 
given by this action of $\text{SL}(2,\mathbb C)$, 
where $\sigma_n$ is the homomorphism 
from $\text{SL}(2,\mathbb C)$ to $\text{GL}(V_n)$.
It is known that each representation $(V_n,\sigma_n)$ 
turns into an irreducible $\text{SL}(n,\mathbb C)$-representation 
of $\text{SL}(2,\mathbb C)$ and that every irreducible 
$n$-dimensional representation of $\text{SL}(2,\mathbb C)$ 
is equivalent to $(V_n,\sigma_n)$.
Composing $\text{Hol}_{(M,\eta)}$ with $\sigma_n$, 
we obtain the representation
$$
\rho_n:
\pi_1(M,\eta)\to\text{SL}(n,\mathbb C).
$$

In the following section, 
we discuss the twisted Alexander invariants and 
Reidemeister torsions associated with this representation 
$\rho_n$. 

\begin{remark}
There have been several studies recently on Reidemeister torsions 
and twisted Alexander polynomials associated 
with an $\text{SL}(n,\mathbb C)$-
representation of fundamental groups.
These are called higher-dimensional Reidemeister torsions
or twisted Alexander polynomials. 
See \cite{tran-yamaguchi1,tran-yamaguchi2,yamaguchi2,yamaguchi3}.
\end{remark}

\section{Results of  Menal-Ferrer and Porti}\label{sec:menal-porti}
In this section, we give a short review of the relevant work 
by Menal-Ferrer and Porti
\cite{menal-porti1, menal-porti2}.

Suppose $M$ is a complete, oriented, hyperbolic 3-manifold 
of finite volume. 
Thus, $M$ is the interior of a compact manifold 
whose boundary consists of tori $T_1,\ldots,T_b$; 
that is, 
$\bdd(\text{cl}(M))=T_1\cup\cdots\cup T_b$. 
In what follows, we often abbreviate $\bdd(\text{cl}(M))$ to $\bdd M$.

Let us consider the homology groups of $M$ 
twisted by $\rho_n$ that is introduced in Section \ref{sec:holonomy}.
Furthermore, let $V_n$ be the vector space stated in Section \ref{sec:holonomy}, 
that is, 
the vector space on which $\rho_n(\gamma)$ acts 
for $\gamma\in\pi_1(M)$.

\begin{proposition}[Corollaries 3.6 and 3.7 in \cite{menal-porti1}]\label{prop:dim}
Suppose $\rho_n$ is associated with an acyclic spin structure. 
\begin{enumerate}
\item
If $n$ is even, then
$\dim_{\mathbb C}H^i(\bdd M;V_{n})=\dim_{\mathbb C}H^i(M;V_{n})=0$ for $i=0,1,2$,
\item
If $n$ is odd, 
then
$\dim_{\mathbb C}H^0(\bdd M;V_{n})=\dim_{\mathbb C}H^2(\bdd M;V_{n})=b$, 
$\dim_{\mathbb C}H^1(\bdd M;V_{n})=2b$, 
$H^0(M;V_{n})=0$ 
and 
$\dim_{\mathbb C}H^i(M;V_{n})=b$
for $i=1,2$. 
\end{enumerate}
\end{proposition}
Moreover, 
Menal-Ferrer and Porti determined a basis of the homology groups. 
Note that Poincar\'e duality with coefficients in $\rho_n$ holds
(Corollary 3.7 in \cite{menal-porti2}).

\begin{proposition}[Proposition 4.6 in \cite{menal-porti2}]\label{prop:basis}
Suppose $n$ $(\ge 3)$ is odd. 
Let $G_{\ell}<\pi_1(M)$ be some fixed realization of the fundamental group of $T_{\ell}$ 
as a subgroup of $\pi_1(M)$. 
For each $T_{\ell}$ choose a non-trivial cycle $\theta_{\ell}\in H_1(T_{\ell};\mathbb Z)$, 
and a non-trivial vector $v_{\ell}\in V_n$ fixed by $\rho_n(G_{\ell})$. 
If $i_{\ell}: T_{\ell}\to M$ denotes the inclusion, 
then the following holds: 
\begin{enumerate}
\item
A basis for $H_1(M;V_{n})$ is given by 
$$
\{i_{1*}([v_1\otimes\theta_1]), \ldots , i_{b*}([v_b\otimes\theta_b])\}.
$$
\item
Let $[T_{\ell}]\in H_2(T_{\ell};\mathbb Z)$ be a fundamental class of $T_{\ell}$. 
Then, a basis for $H_2(M;V_{n})$ is given by 
$$
\{i_{1*}([v_1\otimes T_1]), \ldots , i_{b*}([v_b\otimes\ T_b])\}.
$$
\end{enumerate}
\end{proposition}
Set 
$\bold h^1
=\{
i_{1*}([v_1\otimes\theta_1]), \ldots , i_{b*}([v_b\otimes\theta_b])
\},\,
\bold h^2
=\{
i_{1*}([v_1\otimes T_1]), \ldots , i_{b*}([v_b\otimes\ T_b])\}$, 
and 
$\bold h=\{\bold h^1,\bold h^2\}$. 
Furthermore, let $\eta$ be a spin structure of $M$.
Then, we define the following quotients:
\begin{align*}
\mathcal T_{2k+1}(M,\eta)
& :=
\frac{\text{Tor}(C_*(M;V_{2k+1}),\bold h)}{\text{Tor}(C_*(M;V_{3}),\bold h)}\hspace{0.3cm}\in\mathbb C^*/\{\pm 1\},\\
\mathcal T_{2k}(M,\eta)
& :=
\frac{\text{Tor}(C_*(M;V_{2k}),\bold h)}{\text{Tor}(C_*(M;V_{2}),\bold h)}
\hspace{0.3cm}\in\mathbb C^*/\{\pm 1\}.
\end{align*}
It is known that 
the quantity $\mathcal T_{2k+1}$ is independent of the spin structure
(Remark \ref{rem:factor}), 
and hence we shall denote it simply by $\mathcal T_{2k+1}(M)$. 
On the other hand, $H_i(M;V_{\rho_{2k}})=0$ for $i=1,2$ 
by Proposition \ref{prop:dim}, then 
we may define as follows: 

\begin{definition}\label{def:torsions}
\begin{align*}
\mathcal T_{2k+1}(M)
& :=
\frac{\text{Tor}(C_*(M;V_{2k+1}),\bold h)}{\text{Tor}(C_*(M;V_{3}),\bold h)}
\hspace{0.3cm}\in\mathbb C^*/\{\pm 1\},\\
\mathcal T_{2k}(M,\eta)
& :=
\frac{\text{Tor}(C_*(M;V_{2k}))}{\text{Tor}(C_*(M;V_{2}))}
\hspace{0.3cm}\in\mathbb C^*/\{\pm 1\}.
\end{align*}
\end{definition}
Note that Proposition 4.2 in \cite{menal-porti2} proves 
that $\mathcal T_{2k+1}(M)$ is independent of 
the choice of $\bold h$.
These invariants are called 
{\it the higher-dimensional Reidemeister torsion invariants}.

We say that a spin structure $\eta$ on $M$ is \textit{positive} on $T_{\ell}$ 
if, for all $\gamma\in \pi_1(T_{\ell})$, 
we have $\text{trace Hol}_{(M,\eta)}(\gamma)=+2$. 
Otherwise, we say that $\eta$ is \textit{non-positive}. 
A spin structure $\eta$ is \textit{acyclic} 
if $\eta$ is non-positive on each $T_{\ell}$.

Theorem 7.1 in \cite{menal-porti2} asserts the following. 

\begin{theorem}[Theorem 7.1 in \cite{menal-porti2}]\label{thm:menal-porti}
Let $M$ be a connected, complete, hyperbolic 3-manifold 
of finite volume. Then 
$$
\lim_{k\to\infty}\frac{\log|\mathcal T_{2k+1}(M)|}{(2k+1)^2}
=
\frac{{\rm Vol}(M)}{4\pi}.
$$
In addition, if $\eta$ is an acyclic spin structure of $M$, 
then 
$$
\lim_{k\to\infty}\frac{\log|\mathcal T_{2k}(M,\eta)|}{(2k)^2}
=
\frac{{\rm Vol}(M)}{4\pi}.
$$
\end{theorem}
As noted in Remark \ref{rmk:sign}, the sign of the right-hand sides 
is positive. 

Let $\nu$ be a simple closed curve in $\bdd M$. 
The holonomy representation can be lifted to $\text{SL}(2,\mathbb C)$, 
but we do not know whether a holonomy representation 
extends to a holonomy representation on $M_{\nu}$, 
where $M_{\nu}$ is the manifold obtained by the Dehn filling 
along $\nu$. 
The problem is crucial 
because Theorem \ref{thm:menal-porti}, 
%a volume formula in \cite{menal-porti2}, 
which we will use and 
is given if $\eta$ is acyclic, is obtained via
Thurston's hyperbolic Dehn surgery theorem.
Corollary 3.4 in \cite{menal-porti2} gives a sufficient condition 
to guarantee the existence of an acyclic spin structure.
The condition is the following:
\begin{equation}\label{eq:menal-porti}
\begin{split}
&\text{
For each boundary component $T_{\ell}$ of $M$, 
the map: }
H_1(T_{\ell};\mathbb Z/2\mathbb Z)\to
H_1(M;\mathbb Z/2\mathbb Z)\\
&\text{
induced by the inclusion has non-trivial kernel.
}
\end{split}
\end{equation}

\begin{remark}
If $\bdd M$ consists of only one torus (i.e., $b=1$),
then $M$ satisfies the condition (\ref{eq:menal-porti});
see Lemma 6.8 in \cite{hempel}, for example. 
\end{remark}

Because we focus on knots and links, we define a family of knots and links 
satisfying the condition (\ref{eq:menal-porti}).
See Chapter 5D in \cite{rolfsen}.

\begin{definition}\label{def:link}
Let $L$ be a $b$-component link in $S^3$. 
One calls $L=K_1\cup \cdots \cup K_b$ an \textit{algebraically split link} if 
the linking numbers $\ell k(K_i,K_j)$ are zero for all $i,\,j$.
\end{definition}

Because of a Seifert surface for a knot, 
a knot becomes an algebraically split link. 
The Whitehead link, the Borromean link and a boundary link 
are examples of algebraically split links.

\section{Homology with local coefficients
corresponding to twisted Alexander invariants}\label{sec:dubois}

As in the previous section, we let $M$ be a complete, 
oriented, hyperbolic 3-manifold of finite volume
whose boundary consists of tori $T_1\cup\cdots\cup T_{b}$.
In Propositions \ref{prop:dim} and \ref{prop:basis},
we studied the homology group $H_*(M;V_n)$. 
In this section, we will investigate $H_*(M;V_n(t))$.
According to the description in Section \ref{sec:pre}, 
let $\alpha$ ($\alpha_h$ resp.) be a surjective homomorphism of 
$\pi_1(M)\to\mathbb Z$  
($H_1(M,\mathbb Z)\to\mathbb Z$ resp.).
and suppose that 
$\alpha\otimes\rho_n:\pi_1(M)\to \text{Aut}(V_n(t))$ 
and 
$C_*(M;V_n(t))
=V_n(t)\underset{\mathbb Z[\pi_1(M)]}{\otimes}C_*(\widetilde{M};\mathbb Z)$
are defined 
as in (\ref{eq:rep}) and (\ref{eq:chain2}).

Let $i$ be the inclusion $\partial M\to M$.
Suppose that the composition 
$\alpha_h\circ (i|_{T_{\ell}})_*:H_1(T_{\ell};\mathbb Z)\cong\mathbb Z\oplus\mathbb Z
\to \mathbb Z$
is nontrivial for any $\ell$.
We choose two loops on $T_{\ell}$, 
say $\lambda_{\ell}$ and $\mu_{\ell}$
such that $[\lambda_{\ell}]$ generates the kernel of 
$\alpha_h\circ (i|_{T_{\ell}})_*$
and $\mu_{\ell}$ intersects $\lambda_{\ell}$ transversely 
in s single point.
The loop $\lambda_{\ell}$ is called the {\it longitude} 
and $\mu_{\ell}$ is called the {\it meridian}.
\begin{equation}\label{eq:dubois-yamaguchi}
\text{Based on these, we denote } \alpha(\mu_{\ell})=t^{a(\ell)}.
\end{equation}

The purpose of this section is to prove the next proposition.
Note that the contents in this section are due to  
Section 3.2.1 in \cite{dubois-yamaguchi1}
and Sections 2 and 3 in \cite{kirk-livingston}.

\begin{proposition}\label{prop:homology1}
$H_k(M;V_n(t))=0$ for any $k$.
\end{proposition}

$M$ has the infinite cyclic cover $\overline{M}$ corresponding to $\alpha$, 
whose boundary consists of 
$b$ component annuli $A_1\cup\cdots\cup A_{b}$.
Then, $\pi_1(\overline{M})=\ker\alpha$ and we have the 
following short exact sequence
\begin{equation*}
1\to\pi_1(\overline{M})\to\pi_1(M)\xrightarrow{\alpha}\mathbb Z\to 1.
\end{equation*}
We denote by the same symbol $\rho_n$ 
the restriction of the representation $\rho_n$ of $\pi_1(M)$
to $\pi_1(\overline{M})$, 
and $\widetilde{M}$ is also the universal cover of $\overline{M}$ 
with covering group $\pi_1(\overline{M})$.
As noted in Section \ref{sec:pre}, $C_*(\widetilde{M};\mathbb Z)$ 
is a left $\mathbb Z[\pi_1(M)]$-module, 
by restriction, also a left $\mathbb Z[\pi_1(\overline{M})]$-module 
on the cells of $\overline{M}$. 
As in (\ref{eq:chain1}), we can form the chain complex
\begin{equation}\notag
C_*(\overline{M};V_n)=V_n\underset{\mathbb Z[\pi_1(\overline{M})]}{\otimes}
C_*(\widetilde{M};\mathbb Z). 
\end{equation}
Furthermore, by the same action as (\ref{eq:module2}), 
we have the chain complex 
\begin{equation}\notag
C_*(M;V_n[t,t^{-1}])=(\mathbb C[t,t^{-1}]
\underset{\mathbb C}{\otimes}V_n)\underset{\mathbb Z[\pi_1(M)]}{\otimes}
C_*(\widetilde{M};\mathbb Z).
\end{equation}
It is known that 
$C_*(\overline{M};V_n)\cong C_*(M;V_n[t,t^{-1}])$
(cf. Theorem 2.1 in \cite{kirk-livingston}), 
and 
there is the natural inclusion 
$C_*(M;V_n[t,t^{-1}])\subset C_*(M;V_n(t))$. 
Moreover, 
the universal coefficient theorem implies that 
$H_*(M;V_n(t))=V_n(t)\underset{\mathbb C[t,t^{-1}]}{\otimes}
H_*(M;V_n[t,t^{-1}])$. 
Thus, we have

\begin{lemma}\label{lem:nofreepart}
$H_*(M;V_n(t))=0$ if and only if 
$H_*(\overline{M};V_n)\cong H_*(M;V_n[t,t^{-1}])$ 
has no free part.
\end{lemma}

\begin{lemma}\label{lem:acyclic0}
$H_0(M;V_n(t))=0.$
\end{lemma}
\begin{proof}
This lemma is obtained from Proposition 3.5 in \cite{kirk-livingston}.
\end{proof}

According to \cite{milnor3}, the short exact sequence 
\begin{equation}\label{eq:milnor}
0\to C_*(\overline{M};V_n)\xrightarrow{t-1}C_*(\overline{M};V_n)
\xrightarrow{t=1}C_*(M;V_n)\to 0
\end{equation}
induces the next long exact sequence in the twisted homology, 
which is called the Wang exact sequence (\cite{spanier}, p456).
\begin{align}\label{eq:exact1}
0&\to H_2(\overline{M};V_n)\xrightarrow{t-1}H_2(\overline{M};V_n)
      \xrightarrow{t=1}H_2(M;V_n)\\
 &\xrightarrow{\delta}H_1(\overline{M};V_n)\xrightarrow{t-1}
      H_1(\overline{M};V_n)\xrightarrow{t=1}H_1(M;V_n)\to\cdots \notag
\end{align}
Here, $\delta$ is the so-called connecting map.

\begin{lemma}\label{lem:acycliceven}
Suppose $n$ is even.
Then, $H_k(M;V_n(t))=0$ for any $k$.
\end{lemma}

\begin{proof}
By Proposition \ref{prop:dim}, 
$H_k(M;V_n)=0$ for $k=0,1,2$. 
From the exact sequence (\ref{eq:exact1}), 
we obtain isomorphisms for $k=0,1,2$: 
$$
0\to H_k(\overline{M};V_n)\xrightarrow[\cong]{t-1} H_k(\overline{M};V_n)\to 0.
$$
This implies that $H_k(\overline{M};V_n)$ has no free part, 
whereupon we have this lemma by Lemma \ref{lem:nofreepart}.
\end{proof}

From now on, we consider the $n$:odd case. 
Applying the short exact sequence (\ref{eq:milnor})
to the boundaries of $M$ and $\overline{M}$, 
namely, $\{T_{\ell}\}_{\ell=1}^{\ell=b}$ 
and 
$\{A_{\ell}\}_{\ell=1}^{\ell=b}$, 
we have
\begin{equation}\label{eq:exact2}
0\to C_*(A_{\ell};V_n)\xrightarrow{t-1}C_*(A_{\ell};V_n)
      \xrightarrow{t=1}C_*(T_{\ell};V_n)\to 0.
\end{equation}
Furthermore,
we have the following commutative diagram 
since all constructions are natural and 
$\oplus_{\ell=1}^bH_2(T_{\ell};V_n)\cong H_2(M;V_n)$ 
by Propositions \ref{prop:dim} and \ref{prop:basis}: 

\begin{equation}\label{eq:commutative1}
\begin{CD}
 0@>>> \displaystyle{\oplus_{\ell=1}^b}H_2(T_{\ell};V_n) 
 @>{\delta_{\partial}}>> \oplus_{\ell=1}^bH_1(A_{\ell};V_n)
 @>{t-1}>> \oplus_{\ell=1}^bH_1(A_{\ell};V_n) 
 @>{t=1}>>\cdots \\
 @. @VV{\cong}V @VV{\iota_*}V @VVV \\
 \cdots@>{t=1}>> H_2(M;V_n) @>{\delta}>> H_1(\overline{M};V_n) @>{t-1}>>
 H_{1}(\overline{M};V_n) @>{t=1}>>\cdots
 \end{CD}
\end{equation}
Here, $\delta_{\partial}$ means the connecting map corresponding to $\delta$ 
and $\iota$ is the inclusion of $A_{\ell}$ into $\overline{M}$.

By Proposition \ref{prop:basis}, 
we may suppose that $H_2(M;V_n)$ is generated by $[v_{\ell}\otimes T_{\ell}]$, 
where $v_{\ell}$ is an invariant vector of $\rho^{-1}_n(\pi_1(T_{\ell}))$. 
Let us consider the mapping $\delta_{\partial}$ in (\ref{eq:commutative1}). 

\begin{lemma}\label{lem:connecting}
$\delta_{\partial}([v_{\ell}\otimes T_{\ell}])
=[(1+t+\cdots+t^{a(\ell)-1})\otimes v_{\ell}\otimes \lambda_{\ell}]$
where $\lambda_{\ell}$ is a longitude of $T_{\ell}$, 
for all $\ell$.
\end{lemma}
\begin{proof}
For ease of notation, 
denote $C_*(A_{\ell};V_n)$ by $C_*$ and $C'_*$, 
and denote $C_*(T_{\ell};V_n)$ by $C''_*$ 
in the exact sequence (\ref{eq:exact2}). 
Thus, we have a commutative diagram:
\begin{equation}\notag
\begin{CD}
0 @>>> C'_2 @>{t-1}>> C_2 
 @>{j_2}>> C''_2 @>>> 0\\
@. @VV{d'_2}V @VV{d_2}V @VV{d_2''}V \\
0 @>>> C'_1 @>{t-1}>> C_1 @>{j_1}>>
 C''_1 @>>> 0 \\
\end{CD}
\end{equation}
Let $z''=v_{\ell}\otimes T_{\ell}\in C''_2$. 
For $z''$, there is an element in $C_2$, say $c_2$, 
such that $j_2(c_2)=z''$ 
since $j_2$ is surjective. 
By the construction, we may suppose
$c_2=1\otimes v_{\ell}\otimes A_{\ell}$. 
Since $[z'']$ is a generator of $H_2(T_{\ell};V_n)$, 
$z''\in \ker d''_2$, then
$j_1\circ d_2(c_2)=d''_2\circ j_2(c_2)=d''_2(z'')=0$. 
Thus, $d_2(c_2)\in \ker j_1=\text{im} (t-1)$, so 
there exists an element $z'\in C'_1$ such that $(t-1)z'=d_2(c_2)$.
Since $d_2(A_{\ell})=\mu_{\ell}\cdot\lambda_{\ell}-\lambda_{\ell}$ 
by the notation (\ref{eq:dubois-yamaguchi}) 
and $v_{\ell}$ is fixed by $\rho_n^{-1}(\mu_{\ell})$, 
we have 
\begin{align}\label{eq:boundaryoperator1}
d_2(c_2)&=d_2(1\otimes v_{\ell}\otimes A_{\ell})
\overset{(\ref{eq:bdd2})}{=}1\otimes v_{\ell}\otimes d_2(A_{\ell})
\overset{(\ref{eq:dubois-yamaguchi})}{=}
1\otimes v_{\ell}\otimes(\mu_{\ell}\cdot\lambda_{\ell}-\lambda_{\ell})\\ \notag
&=
1\otimes v_{\ell}\otimes \mu_{\ell}\cdot\lambda_{\ell}-1\otimes v_{\ell}\otimes \lambda_{\ell}
\overset{(\ref{eq:module2})}{=}
\alpha(\mu_{\ell})\otimes \rho^{-1}_n(\mu_{\ell})(v_{\ell})\otimes\lambda_{\ell}-1\otimes v_{\ell}\otimes \lambda_{\ell}\\
&\overset{(\ref{eq:dubois-yamaguchi})}{=}
t^{a(\ell)}\otimes v_{\ell}\otimes\lambda_{\ell}-1\otimes v_{\ell}\otimes \lambda_{\ell}
=
(t^{a(\ell)}-1)(1\otimes v_{\ell}\otimes\lambda_{\ell}). \notag
\end{align}
Hence, $z'=(1+t+\cdots+t^{a(\ell)-1})\otimes v_{\ell}\otimes\lambda_{\ell}$.
Thus, we have this lemma.
(We omit the well-definedness.) 
\end{proof}

Since $M$ is a 3-manifold with tori boundary, 
the Euler number $\chi(M)=0$. 
Then, the universal coefficient theorem and the Euler-Poincar\'e theorem 
imply that 
$\sum_{d=0}^2(-1)^d\text{rk} H_d(\overline{M};V_n)=\chi(M)\times n$. 
Thus, we have 
$\text{rk}H_1(\overline{M};V_n)=\text{rk}H_2(\overline{M};V_n)$
since $\text{rk}H_0(\overline{M};V_n)=0$ 
(Lemmas \ref{lem:nofreepart} and \ref{lem:acyclic0}).

The next lemma with Lemmas \ref{lem:acyclic0} and \ref{lem:acycliceven}
concludes Proposition \ref{prop:homology1}. 
\begin{lemma}\label{lem:acyclicodd}
Suppose $n$ is odd. Then, 
$H_k(M;V_n(t))=0$ for $k=1,2$.
\end{lemma}

\begin{proof}
Suppose that $\text{rk}H_1(\overline{M};V_n)=\text{rk}H_2(\overline{M};V_n)=r (>0)$.
From the exact sequence (\ref{eq:exact1}), 
we obtain:
$\text{im}(t=1)\cong H_2(\overline{M};V_n)/\ker(t=1)
				\cong H_2(\overline{M};V_n)/\text{im}(t-1)
				\cong H_2(\overline{M};V_n)/(t-1)H_2(\overline{M};V_n)
				\ncong 0$. 
Therefore, $\ker\delta\cong\text{im}(t=1)\ncong 0$.
Let $\xi$ be nonzero homology class in $H_2(M;V_n)$ 
such that $\delta(\xi)=0\in H_1(\overline{M};V_n)$. 
Then, we can write that 
$\xi =\sum_{\ell=1}^b\beta_{\ell}[v_{\ell}\otimes T_{\ell}]$ 
by Proposition \ref{prop:basis}, 
and we may regard it as an element in $H_2(T_{\ell};V_n)$. 
By Lemma \ref{lem:connecting}, we have 
$\delta_{\partial}(\xi)
=\delta_{\partial}(\sum_{\ell=1}^b\beta_{\ell}[v_{\ell}\otimes T_{\ell}])
=\sum_{\ell=1}^b\beta_{\ell}[(1+t+\cdots+t^{a(\ell)-1})\otimes v_{\ell}\otimes \lambda_{\ell}]$. 
On the other hand, 
the image of $[(1+t+\cdots+t^{a(\ell)-1})\otimes v_{\ell}\otimes\lambda_{\ell}]$ by the map $(t=1)$ 
is nonzero in $H_1(T_{\ell};V_n)$ and $H_{1}(M;V_n)$ 
by Propositions \ref{prop:dim} and \ref{prop:basis}.
Therefore, $[(1+t+\cdots+t^{a(\ell)-1})\otimes v_{\ell}\otimes\lambda_{\ell}]\neq 0$ 
in $H_1(A_{\ell};V_n)$ 
and $\iota_{*}[(1+t+\cdots+t^{a(\ell)-1})\otimes v_{\ell}\otimes\lambda_{\ell}]\neq 0$
in $H_1(\overline{M};V_n)$ 
from the commutative diagram (\ref{eq:commutative1}).
This contradicts $\delta(\xi)=0$, 
and we have $\text{rk}H_1(\overline{M};V_n)=\text{rk}H_2(\overline{M};V_n)=0$.
Thus, we have this lemma by Lemma \ref{lem:nofreepart}.
\end{proof}

\section{Relationship between twisted Alexander invariants 
and higher-dimensional Reidemeister torsion invariants}\label{sec:yamaguchi}

According to Definition \ref{def:torsions}, 
we define the Reidemeister torsion twisted by 
$\rho_n$ stated in Section \ref{sec:holonomy} 
with preferred bases $\bold h=\{\bold h^1,\bold h^2\}$ 
obtained in Proposition \ref{prop:basis}.
Let $M$ be a complete, oriented, hyperbolic 3-manifold 
whose boundary consists of tori $T_1\cup\cdots\cup T_b$.
One has the meridian $\mu_{\ell}$ and a longitude $\lambda_{\ell}$ 
on $T_{\ell}$ for each boundary torus $T_{\ell}$ 
as in just before (\ref{eq:dubois-yamaguchi}). 
Choose the longitude $\lambda_{\ell}$ as a non-trivial cycle 
in $H_1(T_{\ell};\mathbb Z)$ in Proposition \ref{prop:basis};
namely, 
let
$\bold h^1_{\bf{\lambda}} 
=\{[v_1\otimes\lambda_1],\ldots,[v_b\otimes\lambda_b]\}$ 
and denote $\bold h_{\bold\lambda}=\{\bold h^1_{\bold{\lambda}},\bold h^2\}$.
Under this condition, we prove the following theorem 
in this section.

\begin{theorem}\label{thm:limit}
For a positive integer $k$, one has
\begin{enumerate}
\item
$\displaystyle{
{\rm Tor}(C_*(M;V_{2k}))=\Delta_{M,\rho_{2k}}(1)};$
\item
$\displaystyle{{\rm Tor}(C_*(M;V_{2k+1}),\bold h_{\lambda})
=\lim_{t\to 1}\frac{\Delta_{M,\rho_{2k+1}}(t)}{\prod_{\ell=1}^{b}(t^{a(\ell)}-1)}.}$
\end{enumerate}
\end{theorem}

This theorem corresponds to a result of Milnor \cite{milnor1} 
in the case in which the representation is trivial, 
Theorem A in \cite{kitano} in the case of $\rho_2$ 
(i.e., $2k=2$ case)
and Theorem 5 in \cite{dubois-yamaguchi1} 
in the case of the adjoint representation of $\rho_2$. 
(This case corresponds to that of $2k+1=3$. See Section \ref{sec:holonomy}.)
Actually, (1) is proved from the map at the chain level 
$C_*(M;V_{2k}(t))\to C_*(M;V_{2k})$ 
induced by evaluation $t=1$ 
since the corresponding homologies 
$H_*(M;V_{2k})$ and $H_*(M;V_{2k}(t))$ vanish 
by Proposition \ref{prop:dim} and Lemma \ref{lem:acycliceven}.
(See also Remark 2.11 in \cite{porti2}.)

From now on, we consider the odd dimensional case: 
$n=2k+1 (k=,1,2,\cdots)$.
The idea of the following arguments is the same as 
that in Section 6 in \cite{dubois-yamaguchi1}.
To prove the theorem, we review two lemmas, 
namely,
the {\it base-changing lemma} 
and the {\it multiplicativity lemma}.
Note that the statements of both 
hold without the assumption on $\bold h^1$.

Let $C_*$ be a chain complex 
and $H_i$ its $i$th homology group $H_i(C_*)$. 
The statement of the base-changing lemma is 
as follows. See Proposition 0.2 in \cite{porti1},
for example. 

\begin{lemma}\label{lem:basechange1}
Suppose $\bold c=\{\bold c^i\}$ and 
$\bold c'=\{\bold c'^i\}$ are bases of $C_*$, and 
$\bold h=\{\bold h^i\}$ and $\bold h'=\{\bold h'^i\}$ 
are bases of $H_i$; then,
$$
{\rm Tor}(C_*,\bold c',\bold h')
=
\prod_{i=0}\Big(\frac{[\bold c'^{i}/\bold c^i])}{[\bold h'^{i}/\bold h^i]}\Big)^{(-1)^i}
{\rm Tor}(C_*,\bold c,\bold h).
$$
In particular, if $C_*$ is acyclic, then 
$$
{\rm Tor}(C_*,\bold c',\emptyset)
=
\prod_{i=0}[\bold c'^{i}/\bold c^i]^{(-1)^i}
{\rm Tor}(C_*,\bold c,\emptyset).
$$
\end{lemma}

Let $0\to V'\xrightarrow{i} V\xrightarrow{j} V''\to 0$ 
be a split short exact sequence of finite-dimensional vector spaces 
and $s$ a section of $j$. 
Thus, $i\oplus s:V'\oplus V''\to V$ is an isomorphism. 
Let 
$\bold b'=\{b'_1,\ldots , b'_p\}, 
 \bold b=\{b_1,\ldots , b_q\},
 \bold b''=\{b''_1,\ldots , b''_r\}$  
be bases of $V',V,V''$, respectively, so $p+r=q$.
We say that the bases $\bold b',\bold b$ and $\bold b''$ 
are {\it compatible} 
if the isomorphism $i\oplus s:V'\oplus V''\to V$ has 
determinant 1 in the bases 
$\bold b'\cup \bold b''
 = (\bold b'_1, \ldots , \bold b'_p,\bold b''_1, \ldots , \bold b''_r)$
 of $V'\oplus V''$ and $\bold b$ of $V$.

Let us review the multiplicativity property of the Reidemeister torsion. 
The equation including their signs is known, 
but we omit the signs since they are not important herein. 
For the details, see Theorem 3.2 in \cite{milnor2} and Lemma 3.4.2 in \cite{turaev1}.
Let $0\to C'_*\to C_* \to C''_* \to 0$ be an exact sequence of chain complexes. 
Assume that $C'_*, C_*$ and $C''_*$ are equipped with 
bases and homology bases. 
Let $\bold c'^i,\bold c^i$ and $\bold c''^i$ denote the preferred bases 
of $C'_*, C_*$ and $C''_*$, respectively, 
and $\bold h', \bold h, \bold h''$ be the their homology bases, respectively. 
Then, the long exact sequence in homology 
$$
\cdots\to H_i(C'_*)\to H_i(C_*)\to H_i(C''_*)\to H_{i-1}(C'_*)\to\cdots
$$
can be considered as an acyclic chain complex $\mathcal H_*$ 
by setting 
$$\mathcal H_{3i}=H_i(C''_*),
 \mathcal H_{3i+1}=H_i(C_*),
 \mathcal H_{3i+2}=H_i(C'_*)$$
with preferred bases, respectively. 
\begin{lemma}\label{lem:basechange2}
If for all $i$, the bases $\bold c'^i,\bold c^i$ and $\bold c''^i$ 
are compatible, then 
$$
{\rm Tor}(C_*,\bold c,\bold h)
=
{\rm Tor}(C'_*,\bold c',\bold h')\cdot
{\rm Tor}(C''_*,\bold c'',\bold h'')\cdot
{\rm Tor}(\mathcal H_*,\{\bold h',\bold h,\bold h''\},\emptyset).
$$
In particular, if $C'_*,C_*$ and $C''_*$ are acyclic, 
then 
$$
{\rm Tor}(C_*,\bold c,\emptyset)
=
{\rm Tor}(C'_*,\bold c',\emptyset)\cdot
{\rm Tor}(C''_*,\bold c'',\emptyset).
$$
\end{lemma}

Let $M$ be a compact hyperbolic 3-manifold 
as in Sections \ref{sec:menal-porti} and \ref{sec:dubois}. 
Let $C_*=C_*(M;V_{2k+1})$ and $C_*(t)=C_*(M;V_{2k+1}(t))$
as defined in Section \ref{sec:pre}; namely, 
$C_*$ and $C_*(t)$ are obtained from a CW-complex 
of $M$; see (\ref{eq:chain1}) and (\ref{eq:chain2}). 
We denote by $\bold c=\{\bold c^{i}\}$ and 
$1\otimes\bold c=\{1\otimes\bold c^{i}\}$ 
the bases of $C_*$ and $C_*(t)$, respectively. 
We define $C'_*$ and $C'_*(t)$ as follows: 
The complex $C'_*$ is the subchain complex of $C_*$, 
which is a lift of the homology group $H_*(M;V_{2k+1})$. 
Note that the boundary operators are all zero by this definition. 
According to Proposition \ref{prop:basis}, 
$C'_2$ is spanned by 
$\{v_{\ell}\otimes T_{\ell}~|~1\le\ell\le b\}$
over $\mathbb C$ and 
$C'_1$ is spanned by $\{v_{\ell}\otimes \lambda_{\ell}~|~1\le\ell\le b\}$
over $\mathbb C$, 
where 
$\lambda_{\ell}$ is a longitude on $T_{\ell}$, 
associated with (\ref{eq:dubois-yamaguchi}).
Set $C'_3=C'_0=0$.
Moreover,
set $\bold c'^2=\{v_1\otimes T_1,\ldots , v_b\otimes T_{b}\}$ and
$\bold c'^1=\{v_1\otimes \lambda_1,\ldots , v_b\otimes \lambda_{b}\}$, 
and we say that $C'_*$ is spanned by $\bf c'$.
Similarly, 
we define that
$C'_2(t)$ is spanned by 
$1\otimes\bold c'^2
=\{1\otimes v_{1}\otimes T_{1},\ldots , 1\otimes v_{b}\otimes T_{b}\}$, 
$C'_1(t)$ is spanned by 
$\bold c'^1(t)
=\{1\otimes v_{1}\otimes \lambda_{1},\ldots , 1\otimes v_{b}\otimes \lambda_{b}\}$. 
Set $C'_3(t)=C'_0(t)=0$ and 
we say that $C'_*(t)$ is spanned by $1\otimes\bold c'$. 
We define $C''_*$ ($C''(t)$, resp.) 
as the quotient complex
$C_*/C'_*$ ($C_*(t)/C'_*(t)$, resp.). 
We endow it with the basis $\bold c''$ 
($1\otimes\bold c''$, resp.). 
Note that we sometimes abbreviate $1\otimes\bold c$ to 
$\bold c(t)$ and $1\otimes\bold c^i$ to 
$\bold c^i(t)$ for simplicity. 
Similarly, 
we abbreviate 
$1\otimes\bold c'$ ($1\otimes\bold c''$, resp.) to 
$\bold c'(t)$ ($\bold c''(t)$, resp.) and 
$1\otimes\bold c'^i$ ($1\otimes\bold c''^i$, resp.) to 
$\bold c'^i(t)$ ($\bold c''^i(t)$, resp.).
From the definition in Section \ref{sec:pre}, 
we obtain 
\begin{equation}\label{eq:baseabb}
\lim_{t\to 1}\bold c^i(t)=\bold c^i,\,\,
\lim_{t\to 1}\bold c'^i(t)=\bold c'^i,\,\,\text{ and } 
\lim_{t\to 1}\bold c''^i(t)=\bold c''^i.
\end{equation}
As in the calculation $(\ref{eq:boundaryoperator1})$, 
the boundary operator $d'_2$ works as follows: 
\begin{equation}\label{eq:boundaryoperator2}
d'_2: 1\otimes v_{\ell}\otimes T_{\ell}
\mapsto
(t^{a(\ell)}-1)\cdot(1\otimes v_{\ell}\otimes\lambda_{\ell}).
\end{equation}
Thus, we have the subchain complex of $C_*(t)$ 
\begin{equation}\label{eq:chain3}
C'_*(t):
0\to C'_3(t) \to C'_2(t)\xrightarrow{d'_2}
C'_1(t)\to C'_{0}(t)\to 0.
\end{equation}
Moreover, we have the following short exact sequences 
of complexes: 
\begin{equation}\label{eq:complex1}
0\to C'_*\to C_*\to C''_*\to 0;
\end{equation}
\begin{equation}\label{eq:complex2}
0\to C'_*(t)\to C_*(t)\to C''_*(t)\to 0.
\end{equation}

\begin{lemma}\label{lem:acyclic1}
$H_i(C''_*)=0$ for $i=0,1,2,3$.
\end{lemma}
\begin{proof}
By definition, we have $H_i(C''_*)=0$ for $i=0,3$.
From the short exact sequence (\ref{eq:complex1}), 
we obtain the next long exact sequence 
$$
\cdots\to H_i(C'_*)\to H_i(C_*)\to 
H_i(C''_*)\to H_{i-1}(C'_*)\to H_{i-1}(C_*)\to\cdots.  
$$
By the definition of $C'_*$, 
we have $H_i(C'_*)\cong H_{i}(C_*)$, 
which leads to the conclusion that $H_i(C''_*)=0$ for $i=1,2$.
\end{proof}

\begin{lemma}\label{lem:acyclic2}
For $i=0,1,2,3$,
the following holds:
{\rm (1)} $H_i(C'_*(t))=0$; 
{\rm (2)} $H_i(C_*(t))=0$; 
{\rm (3)} $H_i(C''_*(t))=0$. 
\end{lemma}

\begin{proof}
(1) Since the map $d'_2$ is invertible, 
we have $H_*(C'(t))=0$. 
(2) This is from Proposition \ref{prop:homology1}.
(3) This can be proved by the long exact sequence in homology, 
which is induced by the short exact sequence (\ref{eq:complex2}).
\end{proof}

Using the exact sequence (\ref{eq:complex2}), 
we endow $C_*(t)$ with the basis 
$\bold c'(t)\cup \bold c''(t)$
obtained by lifting and concatenation.
Here,  
we compare the Reidemeister torsions of a hyperbolic 
3-manifold $M$ associated with the basis 
$\bold c'(t)\cup \bold c''(t)$, 
$\bold c'(t)$ and $\bold c''(t)$. 
We write 
$\text{Tor}(C_*(t),\bold c'(t)\bold c''(t),\emptyset)$
($\text{Tor}(C'_*(t),\bold c'(t),\emptyset),
  \text{Tor}(C''_*(t),\bold c''(t),\emptyset)$, resp.) 
the Reidemeister torsion of $C_*(t)$ ($C'_*(t), C''_*(t)$, resp.)
computed in the basis 
$\bold c'(t)\cup\bold c''(t)$
($\bold c'(t), \bold c''(t)$, resp.). 
On the other hand, 
$\text{Tor}(C_*(t),\bold c(t),\emptyset)\\
=\text{Tor}(C_*(M;V_{2k+1}(t)),\bold c(t),\emptyset)$ 
is the Reidemeister torsion of the chain complex 
$C_*(t)$ endowed with the basis $\bold c(t)$, 
which is the twisted Alexander invariant $\Delta_{M,\rho_{2k+1}}(t)$
(see Definition \ref{def:TAI}). 
Note that $\bold c$ is different from $\bold c'\cup\bold c''$.

\begin{lemma}\label{lem:basechange3}
$${\rm Tor}(C_*(t),\bold c(t),\emptyset)
=
{\rm Tor}(C_*(t),\bold c'(t)\bold c''(t),\emptyset)
\cdot\prod_{i=0}^2
[\bold c^i(t)/\bold c'^i(t)\bold c''^i(t)]^{(-1)^i}.
$$ 
\end{lemma}
\begin{proof}
We obtain this lemma from Lemma \ref{lem:basechange1}.
\end{proof}

By Lemma \ref{lem:basechange2}, we have the next lemma. 
Note that the bases are compatible by definition.

\begin{lemma}\label{lem:basechange4}
$$
{\rm Tor}(C_*(t),\bold c'(t)\bold c''(t),\emptyset)
=
{\rm Tor}(C'_*(t),\bold c'(t),\emptyset)
\cdot
{\rm Tor}(C''_*(t),\bold c''(t),\emptyset).
$$
\end{lemma}

\begin{lemma}\label{lem:basechange5}
$${\rm Tor}(C'_*(t),\bold c'(t),\emptyset)
=
\prod_{\ell=1}^{b}(t^{a(\ell)}-1).
$$
\end{lemma}

\begin{proof}
From the chain complex (\ref{eq:chain3})
(and (\ref{eq:boundaryoperator2})), 
we obtain:
$
{\rm Tor}(C'_*(t),\bold c'(t),\emptyset)
=
\det d'_2
=
\prod_{\ell=1}^{b}(t^{a(\ell)}-1).$
\end{proof}

\begin{lemma}\label{lem:basechange6}
\begin{align*}
{\rm Tor}(C''_*(t),\bold c''(t),\emptyset)
=&
[1\otimes d_3(\bold b^3)\bold c'^2(t)1\otimes \bold b^2
/\bold c'^2(t)\bold c''^2(t)]^{-1}\\
\cdot
&[1\otimes d_2(\bold b^2)\bold c'^1(t)1\otimes \bold b^1
/\bold c'^1(t)\bold c''^1(t)]\cdot
[1\otimes d_1(\bold b^1)
/\bold c''^0(t)]^{-1}.
\end{align*}
\end{lemma}
\begin{proof}
By the definition of the way to choose a set of vectors 
$\bold b''^{i+1}$ in $C''_{i+1}$, 
$d''_{i+1}\bold b''^{i+1}$ is a basis of 
$B''_{i}=\text{im}(d''_{i+1}:C''_{i+1}\to C''_{i}).$
Furthermore, we have
\begin{align*}
{\rm Tor}(C''_*(t),\bold c''(t),\emptyset)
&=
\prod_{i=0}^{3}[d''_{i+1}(1\otimes\bold b''^{i+1})1\otimes\bold b''^i/\bold c''^i(t)]^{(-1)^{i+1}}\\
&=
\prod_{i=0}^{2}[d''_{i+1}(1\otimes\bold b''^{i+1})1\otimes\bold b''^i/\bold c''^i(t)]^{(-1)^{i+1}}.
\end{align*}
since $C''_*(t)$ is acyclic (Lemma \ref{lem:acyclic2} (3)) 
and $\bold b''^3=\bold c''^3$.
For $i=1,2$, 
the set of vectors $1\otimes \bold b''^i$ in $C''_i(t)$ 
generates a subspace on which the boundary operator
$d_i:C_i(t)\to C_{i-1}(t)$ is injective by definition, 
so 
$\bold c'^i(t)\cup d_{i+1}(1\otimes\bold b^{i+1})
\cup 1\otimes\bold b^i$ 
may be regarded as a basis of $C_*(t)$, 
where $1\otimes\bold b^i$
is a lift of $1\otimes\bold b''^i$ to $C_*(t)$. 
Therefore, we have
$$
\prod_{i=0}^{2}[d''_{i+1}(1\otimes\bold b''^{i+1})1\otimes\bold b''^i/\bold c''^i(t)]^{(-1)^{i+1}}
=
\prod_{i=0}^{2}[\bold c'^i(t)1\otimes d_{i+1}(\bold b^{i+1})
1\otimes\bold b^i/\bold c'^i(t)\bold c''^i(t)]^{(-1)^{i+1}}.
$$
This is the desired conclusion up to $\pm 1$.
\end{proof}

\noindent {\bf Proof of Theorem \ref{thm:limit}.}

First, we note that
if $t$ tends to 1, the boundary operator $d'_2$ in (\ref{eq:chain3}) 
becomes the zero map and
the complex $C'_*(t)$ changes into $C'_*$. 
Moreover, 
if we suppose $\bold c'^i$ in $C_*$ 
using the exact sequence $(\ref{eq:complex1})$, 
we have the following by Proposition \ref{prop:basis}:
\begin{align}\label{eq:basechange7}
H_1(M;V_{2k+1}) \text{ is generated by } 
\bold h^1_{\bold{\lambda}}&=\{[v_1\otimes\lambda_1],\ldots, [v_b\otimes\lambda_b]\}
\text{ and }
\bold{\tilde{h}_{\lambda}}^1= \bold c'^1;\\
H_2(M;V_{2k+1}) \text{ is generated by } 
\bold h^2&=\{[v_1\otimes T_1],\ldots, [v_b\otimes T_b]\}
\text{ and }
\bold{\tilde{h}}^2= \bold c'^2 \notag.
\end{align}
Here, $\bold{\tilde{h}_{\lambda}}^1$ and $\bold{\tilde{h}}^2$
are lifts of $\bold{h_{\lambda}^1}$ and $\bold{h}^2$ into $C_*$.

Let $\bold b''^{i+1}$ be a set of vectors in $C''_{i+1}$ 
such that $d''_{i+2}(\bold b''^{i+1})$ is a basis
of $B_i=\text{im}(d''_{i+1}:C''_{i+1}\to C''_{i})$.
By Lemmas \ref{lem:basechange1} and \ref{lem:basechange6},  
we have
\begin{align*}
{\rm Tor}(C''_*(t),\bold c''(t),\emptyset)
&=
\{[1\otimes d_3(\bold b^3)\bold c'^2(t)1\otimes \bold b^2
/\bold c^2(t)]\cdot
[\bold c^2(t)/\bold c'^2(t)\bold c''^2(t)]\}^{-1}\\
&\hphantom{=}\cdot
\{[1\otimes d_2(\bold b^2)\bold c'^1(t)1\otimes \bold b^1
/\bold c^1(t)]\cdot
[\bold c^1(t)/\bold c'^1(t))\bold c''^1(t)]\}\\
&\hphantom{=}\cdot
\{[1\otimes d_1(\bold b^1)/\bold c^0(t)]\cdot
[\bold c^0(t)/\bold c''^0(t)]\}^{-1}\\
=\big(&\prod_{i=0}^{2}[\bold c^i(t)/\bold c'^i(t)\bold c''^i(t)]^{(-1)^{i+1}}\big)
\cdot[1\otimes d_3(\bold b^3)\bold c'^2(t)1\otimes \bold b^2
/\bold c^2(t)]^{-1}\\
&\cdot
[1\otimes d_2(\bold b^2)\bold c'^1(t)1\otimes \bold b^1
/\bold c^1(t)]
\cdot
[1\otimes d_1(\bold b^1)/\bold c^0(t)]^{-1}
\end{align*}
Therefore, by Lemmas \ref{lem:basechange3}, \ref{lem:basechange4} and 
\ref{lem:basechange5}, 
we have
\begin{align*}
{\rm Tor}(C_*(t),\bold c(t),\emptyset)
&=
\prod_{\ell=1}^{b}(t^{a(\ell)}-1)\cdot 
{\rm Tor}(C''_*(t),\bold c''(t),\emptyset)
\cdot\prod_{i=0}^2
[\bold c^i(t)/\bold c'^i(t)\bold c''^i(t)]^{(-1)^i}\\
&=
\prod_{\ell=1}^{b}(t^{a(\ell)}-1)\cdot 
[1\otimes d_3(\bold b^3)\bold c'^2(t)1\otimes \bold b^2
/\bold c^2(t)]^{-1}\\
&\hphantom{=}\cdot
[1\otimes d_2(\bold b^2)\bold c'^1(t)1\otimes \bold b^1
/\bold c^1(t)]
\cdot
[1\otimes d_1(\bold b^1)/\bold c^0(t)]^{-1}.
\end{align*}
Moreover, by (\ref{eq:baseabb}) and (\ref{eq:basechange7}) 
together with
$\displaystyle{\lim_{t\to 1}(1\otimes d_3(\bold b^3))=d_3\bold b^3}$ 
and
$\displaystyle{\lim_{t\to 1}(1\otimes \bold b^2)=\bold b^2}$, 
we obtain 
\begin{align*}
\lim_{t\to 1}
&\Big([1\otimes d_3(\bold b^3)\bold c'^2(t)1\otimes \bold b^2
/\bold c^2(t)]^{-1}
\cdot
[1\otimes d_2(\bold b^2)\bold c'^1(t)1\otimes \bold b^1
/\bold c^1(t)]\\
&\hphantom{=}\cdot
[1\otimes d_1(\bold b^1)/\bold c^0(t)]^{-1}\Big)\\
&=
[d_3(\bold b^3)\bold{\tilde{h}^2}\bold b^2)/\bold c^2]^{-1}
\cdot
[d_2(\bold b^2)\bold{\tilde{h}^1_{\lambda}}\bold b^1/\bold c^1]
\cdot
[d_1(\bold b^1)/\bold c^0]^{-1}\\
&=
{\rm Tor}(C_*(M;V_{2k+1}),\bold h_{\bold{\lambda}}).
\end{align*}
By Definition \ref{def:TAI},
$\Delta_{M,\rho_{2k+1}}(t)
=
{\rm Tor}(C_*(M;V_{2k+1}(t)),\bold c(t),\emptyset)
=
{\rm Tor}(C_*(t),\bold c(t),\emptyset)
$, 
so we complete the proof of (2) in Theorem \ref{thm:limit}.

\section{Main results}\label{sec:results}

Let $M$ be a complete, oriented, hyperbolic 3-manifold 
whose boundary consists of torus cusps. 
Set
\begin{equation}\label{eq:correction1}
\mathcal A_{M,2k}(t)=\frac{\Delta_{M,\rho_{2k}}(t)}{\Delta_{M,\rho_{2}}(t)},
\hspace{0.5cm}
\mathcal A_{M,2k+1}(t)=\frac{\Delta_{M,\rho_{2k+1}}(t)}{\Delta_{M,\rho_{3}}(t)}.
\end{equation}
Then, by Theorems \ref{thm:menal-porti} and \ref{thm:limit}, we have
\begin{theorem}\label{thm:main}
Under the above notation, we have:
$$
\lim_{k\to\infty}\frac{\log|\mathcal A_{M,2k+1}(1)|}{(2k+1)^2}
=
\frac{{\rm Vol}(M)}{4\pi}.
$$
If $M$ satisfies the condition (\ref{eq:menal-porti}), we have:
$$
\lim_{k\to\infty}\frac{\log|\mathcal A_{M,2k}(1)|}{(2k)^2}
=
\frac{{\rm Vol}(M)}{4\pi}.
$$
\end{theorem}
If $L$ is an algebraically split link in $S^3$, 
the complement of $L$ satisfies the condition 
(\ref{eq:menal-porti});
see Definition \ref{def:link}.
Hence, we have Theorem \ref{thm:link}.
Since every knot is algebraically split, 
we have the following corollary.

\begin{corollary}[Theorem 1.1 in \cite{goda}]\label{cor:goda}
For any hyperbolic knot $K$, 
$$
\lim_{k\to\infty}\frac{\log|\mathcal A_{K,2k}(1)|}{(2k)^2}
=
\lim_{k\to\infty}\frac{\log|\mathcal A_{K,2k+1}(1)|}{(2k+1)^2}
=
\frac{{\rm Vol}(K)}{4\pi}.
$$
\end{corollary}

\begin{remark}
As in (\ref{eq:correction1}), 
$\mathcal A_{M,\rho_n}(t)$ is defined by dividing the principal part. 
But it is not essential, especially in the even case.
We may describe as follows without a correction term: 
$\Delta_{M,\rho_{2}}(t)$ and $\Delta_{M,\rho_{3}}(t).$
\begin{itemize}
\item
$\displaystyle{\lim_{k\to\infty}\frac{\log|\Delta_{M,2k}(1)|}{(2k)^2}=\frac{{\rm Vol}(M)}{4\pi};}$
\item
$\displaystyle{\lim_{k\to\infty}
\frac{1}{(2k+1)^2}\log\Big|\lim_{t\to 1}
\frac{\Delta_{M,2k+1}(t)}{\prod_{\ell=1}^{b}(t^{a(\ell)}-1)}\Big|=\frac{{\rm Vol}(M)}{4\pi}}$, 
where $b$ is the number of the component of $\partial M$.
\end{itemize}
\end{remark}

\medskip

\noindent {\bf Proof of Theorem \ref{thm:main}.}\\
\noindent Case 1 (even-dimensional representation $\rho_{2k}$ case). 

Since $M$ satisfies the condition (\ref{eq:menal-porti}),
all spin structures on $M$ are acyclic
by Corollary 3.4 in \cite{menal-porti2}. 
Then, we have the conclusion by Theorem \ref{thm:menal-porti} 
and (1) in Theorem \ref{thm:limit}.\\

\noindent Case 2 (odd-dimensional representation $\rho_{2k+1}$ case).

For the simplicity, we assume that 
$a(1)=a(2)=\cdots =a(\ell)=1$. 
We can prove the general case by the same argument.
Part (2) of Theorem \ref{thm:limit} implies that 
$\Delta_{M,\rho_{2k+1}}(t)
=(t-1)^b\tilde{\Delta}_{M,\rho_{2k+1}}(t)$ 
and 
Tor$(C_*(M;V_{2k+1}),\bold h_{\bold{\lambda}})
=\tilde{\Delta}_{M,\rho_{2k+1}}(1)$,  
where $\tilde{\Delta}_{M,\rho_{2k+1}}(t)$ is a rational function. 
Then,  
$$
\mathcal A_{M,2k+1}(1)
=
\frac{\tilde{\Delta}_{M,\rho_{2k+1}}(1)}{\tilde{\Delta}_{M,\rho_{3}}(1)}
=
\frac{{\rm Tor}(C_*(M;V_{2k+1}),\bold h_{\bold{\lambda}})}
{{\rm Tor}(C_*(M;V_{3}),\bold h_{\bold{\lambda}})}
=
\mathcal T_{2k+1}(M)
$$
by Definition \ref{def:torsions}.
Hence, we have Theorem \ref{thm:main} 
by Theorem \ref{thm:menal-porti}. 
\begin{flushright}
$\square$
\end{flushright}

\begin{corollary}\label{cor:-1}
$$
\lim_{k\to\infty}\frac{\log|\mathcal A_{M,2k}(-1)|}{(2k)^2}
=
\lim_{k\to\infty}\frac{\log|\mathcal A_{M,2k+1}(-1)|}{(2k+1)^2}
=
\frac{{\rm Vol}(M)}{4\pi}.
$$
\end{corollary}
\begin{proof}
Let $\hat{M}$ be the 2-fold cyclic cover of $M$, 
and let prj be the covering map $\hat{M}\to M$. 
Then, we have the following exact sequence: 
$$
1 \to\pi_1(\hat{M})\xrightarrow{\text{prj}_*}
\pi_1(M)\xrightarrow{\pi}\mathbb Z/2\mathbb Z\to 1.
$$
Here, we use the pull-back of homomorphisms of $\pi_1(M)$ 
as homomorphisms of $\pi_1(\hat{M})$. 
Let $\alpha$ be the surjective homomorphism 
$\pi_1(M)\to \mathbb Z=\langle t \rangle$ 
and 
$\hat{\alpha}$ the pull-back by $\text{prj}_{*}$. 
We also suppose that $\pi$ factors through $\alpha$, 
and we use the symbol $\hat{\rho}_n$ for 
the pull-back of $\rho_n$ by $\text{prj}_{*}$.
Under this situation, we can apply Corollary 4.4 in \cite{dubois-yamaguchi2}
where $q=2$ and $s=t^2$ to our setting.
Then, we have
$$\Delta_{\hat{M},\rho_n}(s)
=
\Delta_{\hat{M},\rho_n}(t^2)
=
\Delta_{M,\rho_n}(t)\cdot\Delta_{M,\rho_n}(-t).
$$
Suppose $n$ is odd (i.e., $n=2k+1$) 
and 
we denote by $\tilde{\Delta}_{M,\rho_{2k+1}}(t)$ 
the fractional function defined in the proof of Theorem \ref{thm:main}. 
Since 
$
\Delta_{\hat{M},\rho_{2k+1}}(t^2)
=
(t^2-1)^b\tilde{\Delta}_{\hat{M},\rho_{2k+1}}(t^2)$
and 
$
\Delta_{M,\rho_{2k+1}}(t)\cdot\Delta_{M,\rho_{2k+1}}(-t)
=
(t-1)^b\tilde{\Delta}_{M,\rho_{2k+1}}(t)\cdot
(-t-1)^b\tilde{\Delta}_{M,\rho_{2k+1}}(-t)$, 
we have 
$
|\tilde{\Delta}_{\hat{M},\rho_{2k+1}}(t^2)|
=
|\tilde{\Delta}_{M,\rho_{2k+1}}(t)||\tilde{\Delta}_{M,\rho_{2k+1}}(-t)|.$
It is known that 
$\text{Vol}(\hat{M})=2 \text{Vol}(M)$, so
we obtain the following from Theorem \ref{thm:main}:
\begin{align*}
\frac{2\text{Vol(M)}}{4\pi}
&=
\frac{\text{Vol}(\hat{M})}{4\pi}
=
\lim_{k\to\infty}\frac{\log|\mathcal A_{\hat{M},2k+1}(1)|}{(2k+1)^2}
=
\lim_{k\to\infty}\frac{1}{(2k+1)^2}
\log\Big|\frac{\tilde{\Delta}_{\hat{M},\rho_{2k+1}}(1)}
{\tilde{\Delta}_{\hat{M},\rho_{3}}(1)}\Big|\\
&=
\lim_{k\to\infty}\frac{\log|\tilde{\Delta}_{\hat{M},\rho_{2k+1}}(1)|}
{(2k+1)^2}
=
\lim_{k\to\infty}\frac{\log|\tilde{\Delta}_{M,\rho_{2k+1}}(1)|
|\tilde{\Delta}_{M,\rho_{2k+1}}(-1)|}
{(2k+1)^2}\\
&=
\lim_{k\to\infty}\frac{\log|\tilde{\Delta}_{M,\rho_{2k+1}}(1)|}
{(2k+1)^2}+
\lim_{k\to\infty}\frac{\log|\tilde{\Delta}_{M,\rho_{2k+1}}(-1)|}
{(2k+1)^2}\\
&=
\frac{\text{Vol}(M)}{4\pi}+
\lim_{k\to\infty}\frac{\log|\tilde{\Delta}_{M,\rho_{2k+1}}(-1)|}
{(2k+1)^2}.
\end{align*}
Therefore, we have the conclusion in the case that $n$ is odd. 
In the same way, we can prove it in the case that $n$ is even.   
\end{proof}

Taking account of the conjecture on the volume and 
the colored Jones polynomial for knots \cite{kashaev, murakami}, 
it might be worth considering the following problem.

\medskip

\begin{problem}\label{pro:Jones}
Does the equation in Corollary \ref{cor:goda} hold by
substituting any root of unity ? 
\end{problem}

\section{On concrete calculations}\label{sec:calculation}
In this section, we present some known facts and methods 
regarding calculating twisted Alexander invariants.
We give concrete calculations that provide approximate values 
of the volumes of the figure eight knot and the Whitehead link 
complements. 

First, we remark on a representation 
that leads to the hyperbolic volume.
Let $G$ be a finitely generated group
and 
set $R(G)=\text{Hom}(G,\text{SL}(2,\mathbb C))$. 
Since $G$ is finitely generated, 
$R(G)$ may be embedded in a product 
$\text{SL}(2,\mathbb C)\times\cdots\times\text{SL}(2,\mathbb C)$
by mapping each representation to the image of a generating set. 
Thus, $R(G)$ is regarded as an affine algebraic set 
whose defining polynomials are induced by the relators 
of a presentation of $G$. 
It is known that the structure is independent of the choice 
of the presentation of $G$ up to isomorphism. 
For a representation $\rho\in R(G)$, 
its \textit{character} is the map $\chi_{\rho}:G\to \mathbb C$ 
defined by  $\chi_{\rho}(\gamma)=\text{trace}(\rho(\gamma))$
for $\gamma\in G$, where trace means the trace of a matrix.
Let $X(G)$ be the set of all characters. 
For a given $\gamma\in G$, 
we define the map $\tau_{\gamma}:X(G)\to\mathbb C$ 
by $\tau_{\gamma}(\chi)=\chi(\gamma)$. 
Then, it is known that $X(G)$ is an affine algebraic set 
that embeds in $\mathbb C^n$ with coordinates 
$(\tau_{\gamma_{1}},\ldots ,\tau_{\gamma_{n}})$ 
for some $\gamma_1,\ldots ,\gamma_n \in G$.
This affine algebraic set is called the \textit{character variety} 
of $G$. 
Note that the set $\{\gamma_1,\ldots ,\gamma_n\}$ may be chosen 
to contain a generating set of $G$ and  
the projection $R(G)\to X(G)$ given by $\rho\mapsto\chi_{\rho}$ 
is surjective.

Let $M$ be a complete hyperbolic 3-manifold.
The character variety $X(\pi_1(M))$ contains the distinguished component
related to the complete hyperbolic structure. 
The component is defined by containing a lift of the holonomy 
representation $\pi_1(M)\to\text{PSL}(2,\mathbb C)$ 
determined by the complete hyperbolic structure. 

Two representations $\rho$ and $\rho'$ are said to be 
\textit{equivalent} if there is an automorphism $\psi$ 
of the representation space such that 
$\rho'(\gamma)=\psi\circ\rho(\gamma)\circ\psi^{-1}$ 
for all $\gamma\in G$.
In general, the equation $\Delta_{M,\rho}(t)=\Delta_{M,\rho'}(t)$ holds 
if $\rho$ and $\rho'$ are equivalent representations 
(Section 3 in \cite{wada}).
If $\rho_2,\rho'_2 \in R(\pi_1(M))$ are irreducible representations 
with $\chi_{\rho_2}=\chi_{\rho'_2}$, then $\rho_2$ is equivalent to $\rho'_2$ 
(Proposition 1.5.2 in \cite{culler-shalen}). 
So, $\rho_n$ is equivalent to $\rho'_n$, where $\rho_n (\rho'_n$, resp.) 
is the irreducible representation to $\text{SL}(n,\mathbb C)$ 
stated in Section \ref{sec:holonomy}, 
so that $\Delta_{M,\rho_n}(t)=\Delta_{M,\rho'_n}(t)$. 
Thus, 
our theorem stated in Section \ref{sec:results} 
holds for the representations 
that are contained in the distinguished component of $X(\pi_1(M))$.

Next, we review spin structures and lifts of $\text{Hol}_M$.
As in Propositions 2.1 and 2.2 in \cite{menal-porti2}, 
there are one-to-one correspondence spin structures of $M$ 
and lifts of $\text{Hol}_M$.  
Let $L$ be a $b$-component algebraically split link
$L=K_1\cup\cdots\cup K_b$ and $M$ the complement of $L$.
From the homology long exact sequence of 
$(M,\bdd M;\mathbb Z)$ 
and
$(M,\bdd M;\mathbb Z/2\mathbb Z)$, 
we have $b_1(M;\mathbb Z)=b_1(M;\mathbb Z/2\mathbb Z)=b$. 
Hence, the number of lifts of the holonomy representation of $M$ 
to $\text{SL}(2,\mathbb C)$ is equal to 
$|H^1(M;\mathbb Z/2\mathbb Z)|=2^b$
by Corollary 2.3 in \cite{menal-porti2}.

%In our setting, 
For each component $K_{\ell}$ of $L$, 
there is the meridian $\mu_{\ell}$ corresponding to $K_{\ell}$ 
and
there are two types of lift of the holonomy representation 
such that trace $\text{Hol}_{(L,\eta)}([\mu_{\ell}])=+2$ 
and trace $\text{Hol}_{(L,\eta')}([\mu_{\ell}])=-2$ 
where $\eta,\eta'$ are corresponding spin structures. 
More precisely, 
for a $b$-component hyperbolic link $L=K_1\cup\cdots\cup K_b$, 
let $G(L)$ be the fundamental group of the exterior of $L$ in $S^3$. 
We suppose that $G(L)$ has the following Wirtinger presentation: 
$$\langle 
x_{1_1},x_{1_2},\ldots,x_{1_{j_1}},
x_{2_1},x_{2_2},\ldots,x_{2_{j_2}},
\ldots , 
x_{b_1},x_{b_2},\ldots,x_{b_{j_b}}
~|~
r_1,r_2,\ldots, r_{s}
\rangle
$$
where 
the generator $x_{\ell_k}$ corresponds to 
the meridian of 
the $\ell$th component $K_{\ell}$ 
and
$\sum_{k=1}^{b}j_k=s+1$. 
Let $\eta$ be a spin structure of $L$. 
Then, we have trace $\text{Hol}_{(L,\eta)}(x_{\ell_k})=+2$ 
or trace $\text{Hol}_{(L,\eta)}(x_{\ell_k})=-2$ 
for all $k$.
Hence,  
we have a sequence of signs that consists of $b$-component 
$(a_1 a_2 \cdots  a_b)$ 
where 
$
\begin{cases}
a_{\ell}=+ &
\text{if }
\text{trace Hol}_{(L,\eta)}(x_{\ell_k})=+2  \\
a_{\ell}=- &
\text{if } 
\text{trace Hol}_{(L,\eta)}(x_{\ell_k})=-2  
\end{cases}.
$
We call this {\it the sign of holonomy lift} 
corresponding to the spin structure $\eta$
and 
denote by $\rho_{2}^{a_1 a_2 \cdots  a_b}$
the holonomy representation with the spin structure 
$\text{Hol}_{(L,\eta)}$: 
$G(L)\to\text{SL}(2,\mathbb C)$.
Moreover, 
we obtain the following representation: 
$$
\rho_{n}^{a_1 a_2 \cdots  a_b}
:
G(L) \to \text{SL}(n,\mathbb C)
$$
by composing $\rho_{2}^{a_1 a_2 \cdots  a_b}$ with $\sigma_n$
stated in Section \ref{sec:holonomy}.
According to (\ref{eq:correction1}),
we set:
\begin{equation}\label{eq:correction2}
\mathcal A^{a_1a_2\cdots a_b}_{L,2k}(t)
=\frac{\Delta_{L,\rho^{a_1a_2\cdots a_b}_{2k}}(t)}
{\Delta_{L,\rho^{a_1a_2\cdots a_b}_{2}}(t)},
\hspace{0.5cm}
\mathcal A^{a_1a_2\cdots a_b}_{L,2k+1}(t)
=\frac{\Delta_{L,\rho^{a_1a_2\cdots a_b}_{2k+1}}(t)}
{\Delta_{L,\rho^{a_1a_2\cdots a_b}_{3}}(t)}.
\end{equation}

Moreover, the following fact is known.
\begin{remark}\label{rem:factor}
Since an odd-dimensional irreducible complex representation 
of $\text{SL}(2,\mathbb C)$ factors through $\text{PSL}(2,\mathbb C)$, 
$\Delta_{M,\rho_n}$ and 
$\text{Tor}(C_*(M;V_n),\bold h)$  
does not change for any sign of the holonomy lift 
$(a_1 a_2 \cdots  a_b)$ if $n$ is odd.  
\end{remark}
Hence, the sign of the holonomy lift is essential 
in the case of the even-dimensional representation.

\begin{example}
Let $K$ be the figure eight knot $4_1$. 
The volume of $K$ is $2.02988\cdots$.
A Wirtinger presentation of $G(K)$ is 
$$G(K)
=
\langle
a,b~|~ab^{-1}a^{-1}ba=bab^{-1}a^{-1}b
\rangle.
$$
Here, $a$ and $b$ correspond to the meridians of $K$, 
as shown in Fig.\ref{fig:figure8}.
\begin{figure}[htbp]
\begin{center}
\includegraphics[width=0.25\textwidth]{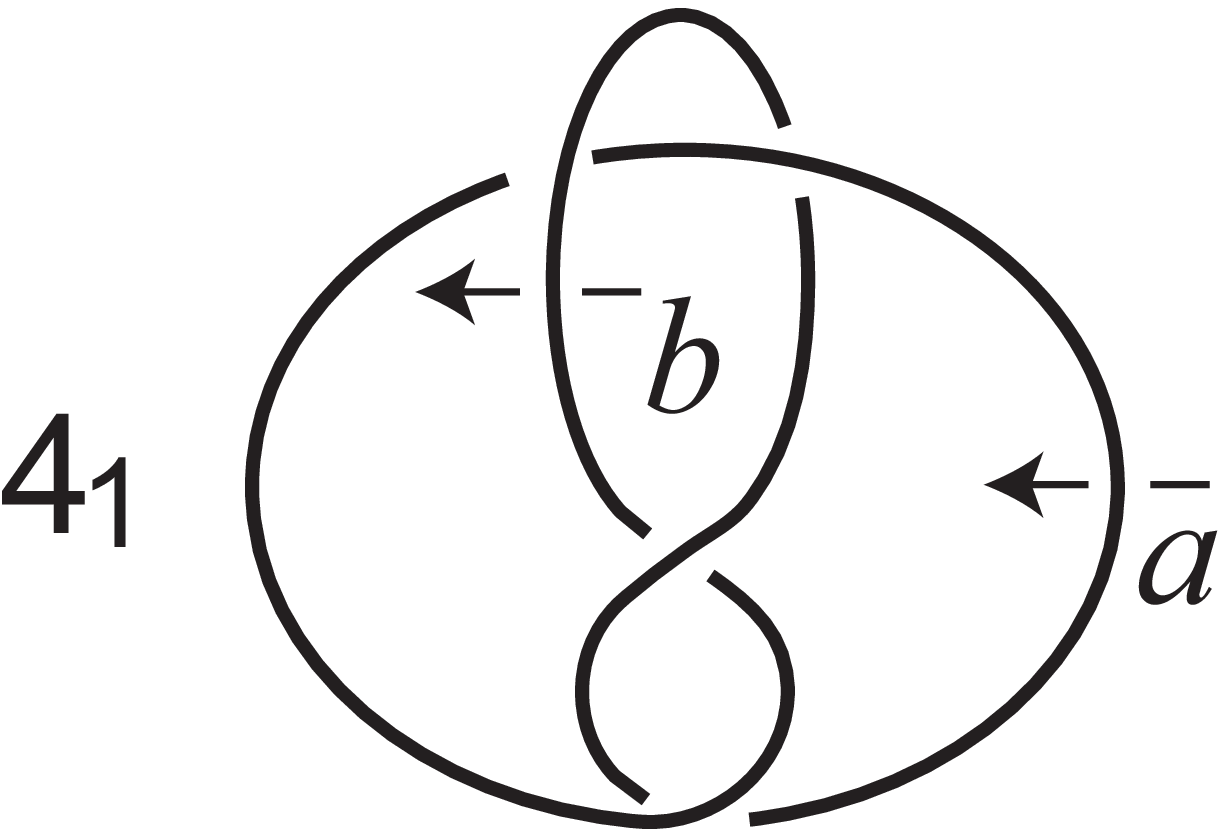}
\end{center}
\caption{}
\label{fig:figure8}
\end{figure}
It is known that the holonomy representation of $G(K)$ 
is given by 
$$
{\text{Hol}}_K(a)
=
\begin{pmatrix}
1 & 1 \\
0 & 1
\end{pmatrix}, 
\,\,\,
{\text{Hol}}_K(b)
=
\begin{pmatrix}
\hphantom{-}1 & 0 \\
-u & 1
\end{pmatrix},
$$
where $u$ is a complex value satisfying 
$u^2+u+1=0$. 
Then, we have
\begin{align*}
\rho_2^+&:G(K)\to\text{SL}(2,\mathbb C):
\rho_2^+(a)
=
\begin{pmatrix}
1 & 1 \\
0 & 1
\end{pmatrix}, 
\,\,\,
\rho^-_2(b)
=
\begin{pmatrix}
\hphantom{-}1 & 0 \\
-u & 1
\end{pmatrix},\\
\rho_2^-&:G(K)\to\text{SL}(2,\mathbb C):
\rho_2^-(a)
=
\begin{pmatrix}
-1 & -1 \\
\phantom{-}0 & -1
\end{pmatrix}, 
\,\,\,
\rho^-_2(b)
=
\begin{pmatrix}
-1 & \hphantom{-}0 \\
\phantom{-}u & -1
\end{pmatrix}.
\end{align*}
Set $A=\rho^+_2(a)$ and $B=\rho^+_2(b)$.
By the definition (Section \ref{sec:holonomy}), 
we have 
$p\big(A^{-1}
\begin{pmatrix}
x\\
y
\end{pmatrix}
\big)
=
p(x-y,y)
$, 
and
$(x-y)^2=x^2-2xy+y^2,\,
 (x-y)y=xy-y^2.
$
Hence, we have
$$
\rho^+_3(a)=
\begin{pmatrix}
\hphantom{-}1 & \hphantom{-}0 & \hphantom{-}0 \\
-2 & \hphantom{-}1 & \hphantom{-}0\\
\hphantom{-}1 & -1 & \hphantom{-}1
\end{pmatrix}.
$$
By the same calculations, we have 
$$
\rho^+_3(b)=
\begin{pmatrix}
1 & u & u^2 \\
0 & 1 & 2u \\
0 & 0 & 1
\end{pmatrix},\,
\rho^+_4(a)=
\begin{pmatrix}
\hphantom{-}1 & \hphantom{-}0 & \hphantom{-}0 & 0 \\
-3 & \hphantom{-}1 & \hphantom{-}0 & 0 \\
\hphantom{-}3 & -2 & \hphantom{-}1 & 0\\
-1 & \hphantom{-}1 & -1 & 1
\end{pmatrix},\,
\rho^+_4(b)=
\begin{pmatrix}
1 & u & u^2 & u^3 \\
0 & 1 & 2u & 3u^2\\
0 & 0 & 1 & 3u\\
0 & 0 & 0 & 1
\end{pmatrix},
\cdots.
$$
Via Fox calculus for $G(K)$, 
we obtain the denominator of 
$\Delta_{K,\rho^+_2}(t)=\det(tB-I)=(t-1)^2$. 
(We have the same results in 
$u=(-1+\sqrt{-3})/2$ and $u=(-1-\sqrt{-3})/2$ 
hereinafter.)
On the other hand, 
the numerator of 
$\Delta_{K,\rho^+_2}(t)=\det(I-t^{-1}AB^{-1}A^{-1}+
AB^{-1}A^{-1}B-tB+BAB^{-1}A^{-1})
=\frac{1}{t^2}(t-1)^2(t^2-4t+1).$
Thus, we obtain the following 
rational functions: 
$$
\Delta_{K,\rho^+_2}(t)=
\frac{1}{t^2}(t^2-4t+1),\,
\Delta_{K,\rho^+_3}(t)=
-\frac{1}{t^3}(t-1)(t^2-5t+1),
$$
$$
\Delta_{K,\rho^+_4}(t)=
\frac{1}{t^4}(t^2-4t+1)^2,\,
\Delta_{K,\rho^+_5}(t)=
-\frac{1}{t^5}(t-1)(t^4-9t^3+44t^2-9t+1).
$$
Using these functions, we have approximations 
to the volume of $4_1$ as follows:
\begin{align*}
\frac{4\pi \log\mathcal |\mathcal A^+_{K,4}(t)|}{4^2}
&=\frac{\pi\log|t^2-4t+1|}{4}\xrightarrow{t=1}\frac{\pi\log 2}{4}\approx 0.544397\cdots, \\
\frac{4\pi \log\mathcal |\mathcal A^+_{K,5}(t)|}{5^2}
&=\frac{4\pi\log|\frac{t^4-9t^3+44t^2-9t+1}{t^2-5t+1}|}{5^2}
\xrightarrow{t=1}
\frac{4\pi\log\frac{28}{3}}{5^2}\approx 1.12273\cdots.
\end{align*}
The data of $\mathcal A^{+}_{K,n}(1)$ for $n\le 15$ can be found in \cite{goda}. 
We proceeded the calculation so that the following data are obtained. 
\begin{table}[htb]
  \begin{tabular}{|c||c|c|c|c|c|} \hline
$n$ & $15$ & $20$ & $25$ & $30$ & $35$ \\ \hline
$\frac{4\pi\log|\mathcal A^{+}_{K,n}(1)|}{n^2}$ 
& $1.93360\cdots$ & $1.97120\cdots$
& $1.99522\cdots$ & $2.00380\cdots$ & $2.01219\cdots$\\ \hline
\end{tabular}
\end{table}
Set 
$\displaystyle{\tilde{\Delta}_{K,\rho_n}(t)=\frac{\Delta_{K,\rho_n}(t)}{t-1}}$ 
if $n$ is odd according to Section \ref{sec:results}.
\begin{table}[htb]
  \begin{tabular}{|c||c|c|c|c|c|} \hline
$n$ & $15$ & $25$ & $35$  \\ \hline
$\frac{4\pi\log|\tilde{\Delta}^{+}_{K,\rho_{n}}(1)|}{n^2}$ 
& $1.99496\cdots$ & $2.01731\cdots$
& $2.02346\cdots$ \\ \hline
$n$ & $20$ & $30$ &  \\ \hline
$\frac{4\pi\log|\Delta^{+}_{K,\rho_{n}}(1)|}{n^2}$ 
& $1.99298\cdots$ & $2.01348\cdots$
& \\ \hline
\end{tabular}
\end{table}

By similar calculations, we have the following data in the case of $\rho^-_2$. 
Note that the value of $\mathcal A^-_{K,n}(1)$ is the same as 
$\mathcal A^+_{K,n}(1)$ in the case of $n$:odd as noted in 
Remark \ref{rem:factor}. 
\begin{table}[htb]
\begin{tabular}{|c||c|c|c|c|c|} \hline
$n$ & $16$ & $20$ & $26$ & $30$ & $36$ \\ \hline
$\frac{4\pi\log|\mathcal A^{-}_{K,n}(1)|}{n^2}$ 
& $1.98381\cdots$ & $2.00039\cdots$
& $2.01243\cdots$ & $2.01677\cdots$ & $2.02078\cdots$\\ \hline
$\frac{4\pi\log|\Delta^{-}_{K,n}(1)|}{n^2}$ 
& $2.07177\cdots$ & $2.05668\cdots$
& $2.04574\cdots$ & $2.04179\cdots$ & $2.03815\cdots$\\ \hline
\end{tabular}
\end{table}
\end{example}

\begin{example}
Let $L$ be the Whitehead link as illustrated in 
Fig. \ref{fig:whitehead}. 
The volume of the complement of $L$ is equal to 
$4\times \text{Catalan's constant}\approx 3.66386\cdots$.
The link group $G(L)$ has the following Wirtinger presentation,  
where $a$ and $b$ correspond to the meridians of $L$ as 
in Fig. \ref{fig:whitehead}: 
$$
G(L)
=
\langle
a,b~|~awa^{-1}w^{-1}
\rangle 
\text{ where } 
w=bab^{-1}a^{-1}b^{-1}ab.
$$

\begin{figure}[htbp]
\begin{center}
\includegraphics[width=0.25\textwidth]{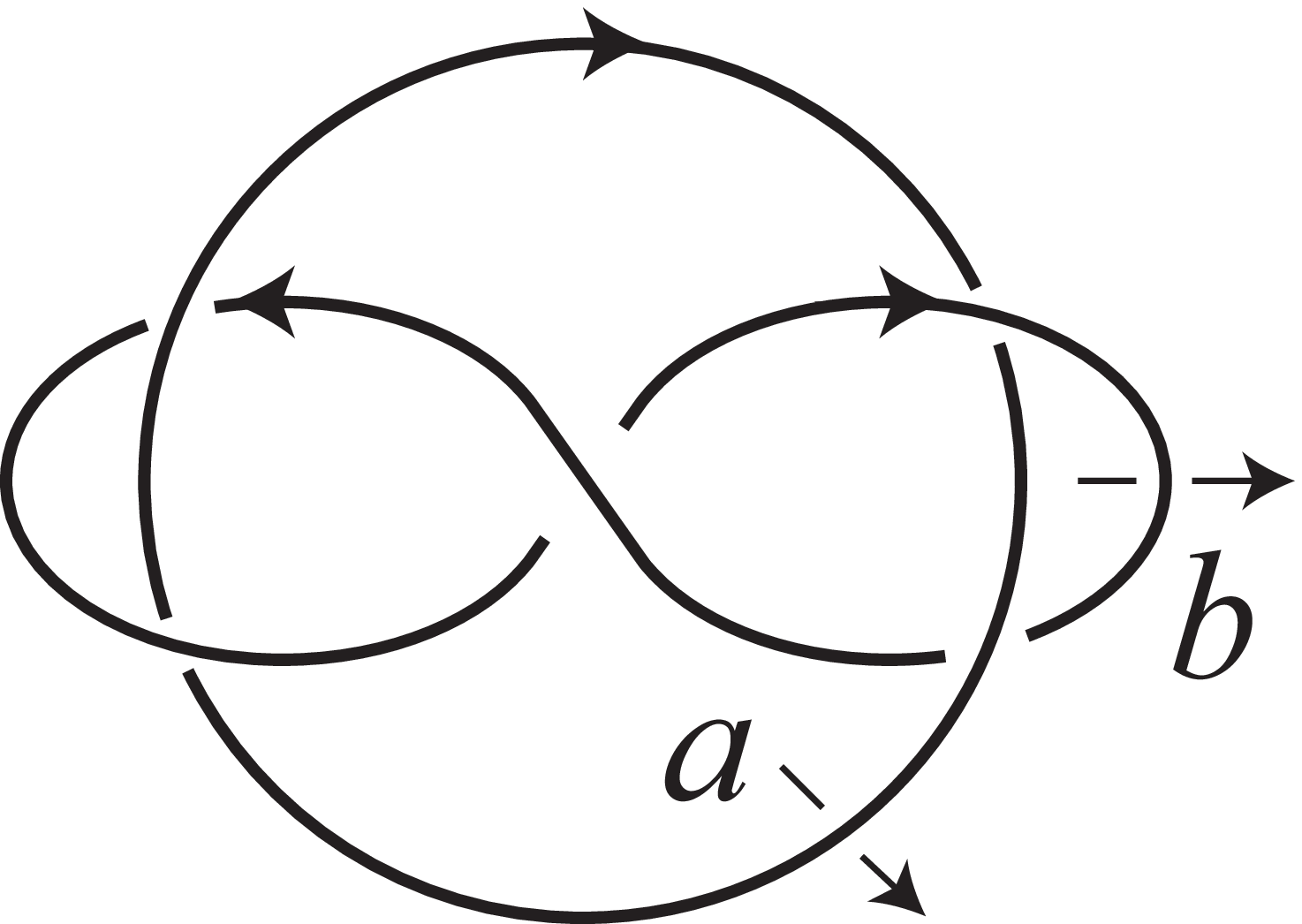}
\end{center}
\caption{}
\label{fig:whitehead}
\end{figure}

Hilden et al.\cite{hilden-lozano-montesinos}
showed an explicit description of the character variety 
of the Whitehead link group. 
In fact, the component of irreducible characters is given by
$$
X(G(L))
=
\{(x,y,v)\in\mathbb C^3~|~p(x,y,v)=0\}\setminus R, 
$$
where 
$p(x,y,v)=xy-(x^2+y^2-2)v+xyv^2-v^3$ 
and 
$R=\{x=\pm 2,v=\pm y\}\cup\{y=\pm 2,v=\pm x\}$;
see Section 4.2 in \cite{dubois-yamaguchi1}.
Let $\rho:G(L)\to \text{SL}(2,\mathbb C)$ be an irreducible 
representation. Then, we can suppose that 
the pair $\rho(a)$ and $\rho(b)$ are expressed as follows 
after taking conjugation with eigenvectors of 
$\rho(a)$ and $\rho(b)$ if necessary: 
$$
\rho(a)=
\begin{pmatrix}
\alpha & 1\\
0 & 1/\alpha
\end{pmatrix}, 
\hspace{0.5cm}
\rho(b)=
\begin{pmatrix}
\beta & 0\\
\gamma & 1/\beta
\end{pmatrix}.
$$
Here, 
the coordinates $(x,y,v)$ as 
$x=\alpha+\alpha^{-1}, y=\beta+\beta^{-1}$ 
and $v=\gamma+\alpha\beta+\alpha^{-1}\beta^{-1}$.
The holonomy representation is obtained by
$\alpha+1/\alpha=\beta+1/\beta=2$ 
since it is a parabolic one.
Therefore, we have four types of representation as follows: 
(I) 
$\rho^{++}_{2}: 
a\mapsto
\begin{pmatrix}
1 & 1 \\
0 & 1
\end{pmatrix},
b\mapsto
\begin{pmatrix}
1 & 0 \\
-1\pm\sqrt{-1} & 1 
\end{pmatrix}
$;
(II)
$\rho^{+-}_{2}: 
a\mapsto
\begin{pmatrix}
1 & 1 \\
0 & 1
\end{pmatrix},
b\mapsto
\begin{pmatrix}
-1 & 0 \\
1\mp\sqrt{-1} & -1 
\end{pmatrix}
$;
(III)
$\rho^{-+}_{2}: 
a\mapsto
\begin{pmatrix}
-1 & -1 \\
\hphantom{-}0 & -1
\end{pmatrix},
b\mapsto
\begin{pmatrix}
1 & 0 \\
-1\pm\sqrt{-1} & 1 
\end{pmatrix}
$;
(IV)
$\rho^{--}_{2}: 
a\mapsto
\begin{pmatrix}
-1 & -1 \\
\hphantom{-}0 & -1
\end{pmatrix},
b\mapsto
\begin{pmatrix}
-1 & \hphantom{-}0 \\
1\mp\sqrt{-1} & -1 
\end{pmatrix}
$. 
%We have the following two-variable twisted Alexander invariants 
%using $\rho^{++}_2,\,\rho^{++}_3$, for example:
%$$
%\Delta_{L,\rho^{++}_2}(t_1,t_2)
%=
%\frac{1}{t_2^2}\Big((t_2-1)^2+t_1^2(t_2-1)^2-2t_1\big(t^2_2-(1\pm\sqrt{-1})t_2+1\big)\Big);
%$$
%$$
%\Delta_{L,\rho^{++}_3}(t_1,t_2)
%=\frac{1}{t_2^3}(t_1-1)(t_2-1)
%\Big((t_2-1)^2+t_1^2(t_2-1)^2-2t_1\big(t_2^2+(2\mp 4\sqrt{-1})t_2+1\big)\Big).
%$$
Then we have: 
$$
\Delta_{L,\rho^{++}_2}(t)
=t^2-4t+(4\pm2\sqrt{-1})-\frac{4}{t}+\frac{1}{t^2};
$$
$$
\Delta_{L,\rho^{++}_3}(t)
=\frac{1}{t^3}(t-1)^2\big(t^4-4t^3-(2\mp8\sqrt{-1})t^2-4t+1\big).
$$
Then, we have the following data by similar calculations to 
figure eight knot case.
\begin{table}[htb]
  \begin{tabular}{|c||c|c|c|c|c|} \hline
$n$ & $10$ & $15$ & $20$ & $25$ & $30$ \\ \hline
$\frac{4\pi\log|\mathcal A^{++}_{L,n}(1)|}{n^2}$ 
& $3.43083\cdots$ & $3.52207\cdots$
& $3.60589\cdots$ & $3.61282\cdots$ & $3.63810\cdots$\\ \hline
\end{tabular}
\end{table}
\begin{table}[htb]
  \begin{tabular}{|c||c|c|c|} \hline
$n$ & $15$ & $25$ &  \\ \hline
$\frac{4\pi\log|\tilde{\Delta}^{++}_{L,n}(1)|}{n^2}$ 
& $3.65757\cdots$ & $3.66160\cdots$ & \\ \hline
$n$ & $10$ & $20$ & $30$ \\ \hline
$\frac{4\pi\log|\Delta^{++}_{L,n}(1)|}{n^2}$ 
& $3.56149\cdots$ & $3.63856\cdots$ & $3.65261\cdots$ \\ \hline
\end{tabular}
\end{table}
\begin{table}
\begin{tabular}{|c||c|c|c|c|c|} \hline
  $n$ & $10$ & $16$ & $20$ & $26$ & $30$ \\ \hline
$\frac{4\pi\log|\mathcal A^{+-}_{L,n}(1)|}{n^2}$ 
& $3.54395\cdots$ & $3.61657\cdots$
& $3.63358\cdots$ & $3.64594\cdots$ & $3.65040\cdots$\\ \hline
$\frac{4\pi\log|\Delta^{+-}_{L,n}(1)|}{n^2}$ 
& $3.67460\cdots$ & $3.66761\cdots$
& $3.66625\cdots$ & $3.66527\cdots$ & $3.66492\cdots$\\ \hline
$\frac{4\pi\log|\mathcal A^{-+}_{L,n}(1)|}{n^2}$ 
& $3.54395\cdots$ & $3.61657\cdots$
& $3.63358\cdots$ & $3.64594\cdots$ & $3.65040\cdots$\\ \hline
$\frac{4\pi\log|\Delta^{-+}_{L,n}(1)|}{n^2}$ 
& $3.67460\cdots$ & $3.66761\cdots$
& $3.66625\cdots$ & $3.66527\cdots$ & $3.66492\cdots$\\ \hline
$\frac{4\pi\log|\mathcal A^{--}_{L,n}(1)|}{n^2}$ 
& $3.45010\cdots$ & $3.58080\cdots$
& $3.61071\cdots$ & $3.63241\cdots$ & $3.64024\cdots$\\ \hline
$\frac{4\pi\log|\Delta^{--}_{L,n}(1)|}{n^2}$ 
& $3.78300\cdots$ & $3.71084\cdots$
& $3.69394\cdots$ & $3.68166\cdots$ & $3.67723\cdots$\\ \hline
\end{tabular}
\end{table}

\end{example}

These calculations were performed using Wolfram Mathematica 
and MathWorks Matlab.

\bibliographystyle{amsplain}

\end{document}